\documentclass[preprint,12pt]{elsarticle}

\usepackage{amssymb}
\usepackage{graphics}
\usepackage{graphicx}
\usepackage{epsfig}

\usepackage{amsthm}

\newcommand{\sysn}{\left\{\begin{array}{rcl}}
\newcommand{\sysk}{\end{array}\right.}

\newcommand{\ingrw}[2]{\includegraphics[width=#1mm]{#2}}

\newtheorem{theorem}{Theorem}[section]

\theoremstyle{example}

\theoremstyle{definition}
\newtheorem{definition}[theorem]{Definition}

\journal{Topology and its Applications}

\begin{document}

\begin{frontmatter}



\title{Classification of selectors for sequences of dense sets of Baire functions}


\author{Alexander V. Osipov}

\ead{OAB@list.ru}


\address{Krasovskii Institute of Mathematics and Mechanics, Ural Federal
 University, \\ Ural State University of Economics, Yekaterinburg, Russia}

\begin{abstract} For a Tychonoff space $X$, we denote by $B(X)$
the space of all Baire functions on $X$ with the topology of
pointwise convergence. In this paper we investigate selectors for
sequences of countable dense and countable sequentially dense sets
of the space $B(X)$. We give the characteristics of selection
principles $S_{1}(\mathcal{P},\mathcal{Q})$ and
$S_{fin}(\mathcal{P},\mathcal{Q})$ for $\mathcal{P},\mathcal{Q}\in
\{\mathcal{D}$, $\mathcal{S}$, $\mathcal{A} \}$, where

$\bullet$ $\mathcal{D}$ --- the family of a countable dense
subsets of $B(X)$;

$\bullet$ $\mathcal{S}$ --- the family of a countable sequentially
dense subsets  of $B(X)$;

$\bullet$ $\mathcal{A}$ --- the family of a countable $1$-dense
subsets of $B(X)$, through the selection principles of $X$.

\end{abstract}


\end{frontmatter}



\section{Introduction}

Many topological properties are defined or characterized in terms
 of the following classical selection principles.
 Let $\mathcal{A}$ and $\mathcal{B}$ be sets consisting of
families of subsets of an infinite set $X$. Then:

$S_{1}(\mathcal{A},\mathcal{B})$ is the selection hypothesis: for
each sequence $\{A_{n}: n\in \omega\}$ of elements of
$\mathcal{A}$ there is a sequence $\{b_{n}: n\in\omega\}$ such
that for each $n$, $b_{n}\in A_{n}$, and $\{b_{n}: n\in\omega \}$
is an element of $\mathcal{B}$.

$S_{fin}(\mathcal{A},\mathcal{B})$ is the selection hypothesis:
for each sequence $\{A_{n}: n\in \omega\}$ of elements of
$\mathcal{A}$ there is a sequence $\{B_{n}: n\in \omega\}$ of
finite sets such that for each $n$, $B_{n}\subseteq A_{n}$, and
$\bigcup_{n\in\omega}B_{n}\in\mathcal{B}$.

$U_{fin}(\mathcal{A},\mathcal{B})$ is the selection hypothesis:
whenever $\mathcal{U}_1$, $\mathcal{U}_2, ... \in \mathcal{A}$ and
none contains a finite subcover, there are finite sets
$\mathcal{F}_n\subseteq \mathcal{U}_n$, $n\in \omega$, such that
$\{\bigcup \mathcal{F}_n : n\in \omega\}\in \mathcal{B}$.

\medskip

In this paper, by cover we mean a nontrivial one, that is,
$\mathcal{U}$ is a cover of $X$ if $X=\cup \mathcal{U}$ and
$X\notin \mathcal{U}$.

 A cover $\mathcal{U}$ of a space $X$ is:

 $\bullet$ an {\it $\omega$-cover} if $X$ does not belong to
 $\mathcal{U}$ and every finite subset of $X$ is contained in a
 member of $\mathcal{U}$.

$\bullet$ a {\it $\gamma$-cover} if it is infinite and each $x\in
X$ belongs to all but finitely many elements of $\mathcal{U}$.

For a topological space $X$ we denote:

$\bullet$ $\mathcal{O}$ --- the family of countable open covers of
$X$;

 $\bullet$ $\Omega$ --- the
family of countable open $\omega$-covers of $X$;

$\bullet$ $\Gamma$ --- the family of countable open
$\gamma$-covers of $X$;

$\bullet$ $\mathcal{B}$ --- the family of countable Baire (Borel
for a metrizable $X$) covers of $X$;

$\bullet$ $\mathcal{B}_{\Omega}$ --- the family of countable Baire
 $\omega$-covers of $X$;

$\bullet$ $\mathcal{B}_\Gamma$ --- the family of countable Baire
 $\gamma$-covers of $X$.



\bigskip

Many equivalence hold among these properties, and the surviving
ones appear in the following The extended Scheepers Diagram (where
an arrow denote implication).

\bigskip
{\begin{figure}
\begin{center}

\ingrw{110}{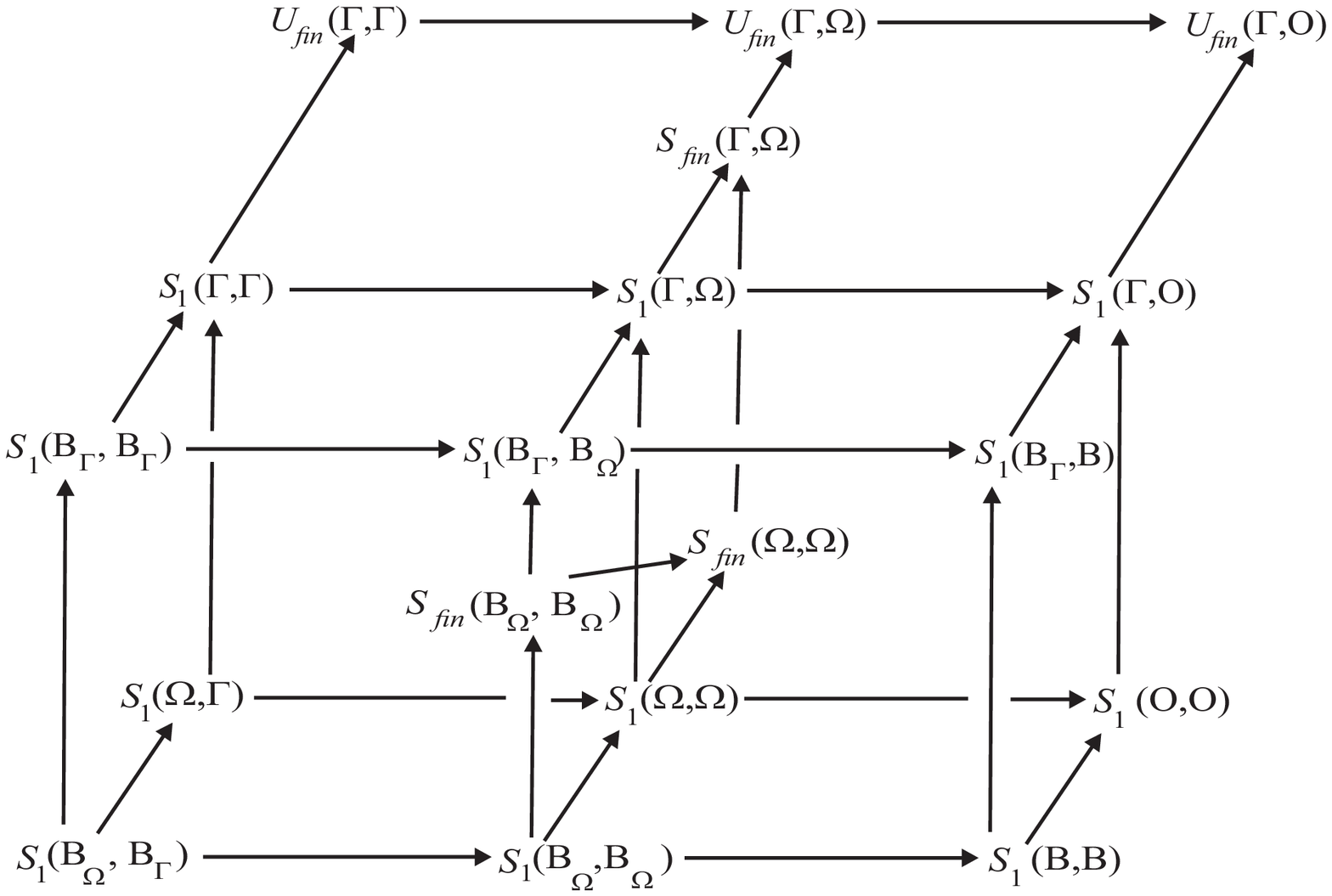}

Figure 1. The extended Scheepers Diagram.

\end{center}

\end{figure}}

 If $X$ is a  space and $A\subseteq X$, then the sequential closure of $A$,
 denoted by $[A]_{seq}$, is the set of all limits of sequences
 from $A$. A set $D\subseteq X$ is said to be sequentially dense
 if $X=[D]_{seq}$. If $D$ is a countable sequentially dense subset
 of $X$ then $X$ call sequentially separable space.

 Call $X$ strongly sequentially separable if $X$ is separable and
 every countable dense subset of $X$ is sequentially dense.
 Clearly, every strongly sequentially separable space is
 sequentially separable, and every sequentially separable space is
 separable.

In paper \cite{osa}, we investigated different selectors for
sequences of dense sets of the space $C_p(X)$ of all real-valued
continuous functions on $X$ with the topology of pointwise
convergence. In this paper we investigate selectors for sequences
of countable dense and countable sequentially dense sets of the
space of all Baire functions, defined on a Tychonoff space $X$. We
give the characteristics of selection principles
$S_{1}(\mathcal{P},\mathcal{Q})$ and
$S_{fin}(\mathcal{P},\mathcal{Q})$ for $\mathcal{P},\mathcal{Q}\in
\{\mathcal{D}$, $\mathcal{S}$, $\mathcal{A} \}$, where

$\bullet$ $\mathcal{D}$ --- the family of a countable dense
subsets of $B(X)$;

$\bullet$ $\mathcal{S}$ --- the family of a countable sequentially
dense subsets  of $B(X)$;

$\bullet$ $\mathcal{A}$ --- the family of a countable $1$-dense
subsets of $B(X)$, through the selection principles of $X$.

\section{Main definitions and notation}

 We will be denoted by

$\bullet$ for any $\alpha\in [0,\omega_1]$, $B_{\alpha}(X)$ a set
of all functions of Baire class $\alpha$, defined
 on a Tychonoff space $X$, provided with the pointwise convergence topology.

In particular:

$\bullet$ for $\alpha=0$, $C_{p}(X)=B_0(X)$ a set of all
real-valued continuous functions $C(X)$ defined
 on a Tychonoff space $X$.

$\bullet$ for $\alpha=1$, $B_1(X)$ a set of all first Baire class
 functions $B_1(X)$ i.e., pointwise limits of continuous functions, defined
 on a Tychonoff space $X$.

$\bullet$ for $\alpha=\omega_1$, $B(X)=B_{\omega_1}(X)$ a  set of
all Baire functions, defined on a Tychonoff space $X$. If $X$ is
metrizable space, then $B(X)$ be called a space of Borel
functions.

Basic open sets of $B_{\alpha}(X)$ are of the form

$[x_1,...,x_k, U_1,...,U_k]=\{f\in B_{\alpha}(X): f(x_i)\in U_i$,
$i=1,...,k\}$, where each $x_i\in X$ and each $U_i$ is a non-empty
open subset of $\mathbb{R}$. Sometimes we will write the basic
neighborhood of the point $f$ as $<f,A,\epsilon>$ where
$<f,A,\epsilon>:=\{g\in B_{\alpha}(X): |f(x)-g(x)|<\epsilon$
$\forall x\in A\}$, $A$ is a finite subset of $X$ and
$\epsilon>0$.

Let $X$ be a topological space, and $x\in X$. A subset $A$ of $X$
{\it converges} to $x$, $x=\lim A$, if $A$ is infinite, $x\notin
A$, and for each neighborhood $U$ of $x$, $A\setminus U$ is
finite. Consider the following collection: $\Omega_x=\{A\subseteq
X : x\in \overline{A}\setminus A\}$ and $\Gamma_x=\{A\subseteq X :
|A|=\aleph_0$ and $x=\lim A\}$. We write $\Pi (\mathcal{A}_x,
\mathcal{B}_x)$ without specifying $x$, we mean $(\forall x) \Pi
(\mathcal{A}_x, \mathcal{B}_x)$.

$\bullet$ A space $X$ has {\it countable fan tightness}
(Arhangel'skii's countable fan tightness), if $X$ $\models$
$S_{fin}(\Omega_x,\Omega_x)$ \cite{arh}.

\medskip

$\bullet$ A space $X$ has {\it countable strong fan tightness}
(Sakai's countable strong fan tightness), if $X$ $\models$
$S_{1}(\Omega_x,\Omega_x)$ \cite{sak}.

\medskip

$\bullet$ A space $X$ has {\it countable selectively sequentially
fan tightness} (Arhangel'skii's property $\alpha_4$), if $X$
$\models$ $S_{fin}(\Gamma_x,\Gamma_x)$ \cite{arh0}.

$\bullet$  A space $X$ has {\it countable strong selectively
sequentially fan tightness} (Arhangel'skii's property $\alpha_2$),
if $X$ $\models$ $S_{1}(\Gamma_x,\Gamma_x)$ \cite{arh0}.

\medskip

$\bullet$ A space $X$ has {\it strictly Fr$\acute{e}$chet-Urysohn
at $x$,} if $X$ $\models$ $S_{1}(\Omega_x,\Gamma_x)$ \cite{sash}.

\medskip

$\bullet$ A space $X$ has {\it almost strictly
Fr$\acute{e}$chet-Urysohn at $x$},  if $X$ $\models$
$S_{fin}(\Omega_x,\Gamma_x)$.

\medskip
$\bullet$ A space $X$ has {\it the weak sequence selection
property}, if $X$ $\models$ $S_{1}(\Gamma_x,\Omega_x)$
\cite{sch4}.

\medskip

$\bullet$  A space $X$ has {\it the sequence selection property},
if $X$ $\models$ $S_{fin}(\Gamma_x,\Omega_x)$.

\medskip
The following implications hold

\medskip
\begin{center}
$S_1(\Gamma_x,\Gamma_x) \Rightarrow S_{fin}(\Gamma_x,\Gamma_x)
\Rightarrow S_1(\Gamma_x,\Omega_x) \Rightarrow
S_{fin}(\Gamma_x,\Omega_x)$ \\  $\Uparrow$ \, \,\, \, \, \, \,
\,\, \, \, \, \, \, $ \Uparrow $ \,\, \, \, \,\, \, \, \, \, \,
$\Uparrow $ \,\, \, \, \,\,\,\, \, \, \, \,\, \, \, $\Uparrow$ \\
$S_1(\Omega_x,\Gamma_x) \Rightarrow S_{fin}(\Omega_x,\Gamma_x)
\Rightarrow S_1(\Omega_x,\Omega_x) \Rightarrow
S_{fin}(\Omega_x,\Omega_x)$

\end{center}

\medskip

We write $\Pi (\mathcal{A}, \mathcal{B}_x)$ without specifying
$x$, we mean $(\forall x) \Pi (\mathcal{A}, \mathcal{B}_x)$.
\medskip

$\bullet$  A space $X$ has {\it countable fan tightness with
respect to countable dense subspaces,} if $X$ $\models$
$S_{fin}(\mathcal{D},\Omega_x)$ (\cite{bbm1}).

$\bullet$ A space $X$ has {\it countable strong fan tightness with
respect to countable dense subspaces,} if $X$ $\models$
$S_{1}(\mathcal{D},\Omega_x)$ (\cite{bbm1}).

$\bullet$ A space $X$ has {\it almost strictly
Fr$\acute{e}$chet-Urysohn at $x$ with respect to countable dense
subspaces}, if $X$ $\models$ $S_{fin}(\mathcal{D},\Gamma_x)$.

$\bullet$ A space $X$ has {\it strictly Fr$\acute{e}$chet-Urysohn
at $x$ with respect to countable dense subspaces}, if $X$
$\models$ $S_{1}(\mathcal{D},\Gamma_x)$.

$\bullet$ A space $X$ has {\it countable selectively sequentially
fan tightness with respect to countable dense subspaces}, if $X$
$\models$ $S_{fin}(\mathcal{S},\Gamma_x)$.

$\bullet$ A space $X$ has {\it countable strong selectively
sequentially fan tightness with respect to countable dense
subspaces}, if $X$ $\models$ $S_{1}(\mathcal{S},\Gamma_x)$.

$\bullet$ A space $X$ has {\it the sequence selection property
with respect to countable dense subspaces}, if $X$ $\models$
$S_{fin}(\mathcal{S},\Omega_x)$.

$\bullet$ A space $X$ has {\it the weak sequence selection
property with respect to countable dense subspaces}, if $X$
$\models$ $S_{1}(\mathcal{S},\Omega_x)$.

\medskip
The following implications hold

\medskip
\begin{center}
$S_1(\mathcal{S},\Gamma_x) \Rightarrow
S_{fin}(\mathcal{S},\Gamma_x) \Rightarrow
S_1(\mathcal{S},\Omega_x) \Rightarrow
S_{fin}(\mathcal{S},\Omega_x)$ \\  \, \, \, $\Uparrow$ \, \, \, \,
\,\,  \, \, \, \, \, $ \Uparrow $ \,\, \, \, \,\, \, \, \, \, \,
$\Uparrow $ \,\, \, \, \,\, \, \,\, \, \, $\Uparrow$
\\ $S_1(\mathcal{D},\Gamma_x) \Rightarrow
S_{fin}(\mathcal{D},\Gamma_x) \Rightarrow
S_1(\mathcal{D},\Omega_x) \Rightarrow
S_{fin}(\mathcal{D},\Omega_x)$

\end{center}

\medskip




\medskip
The following implications hold for countable dense and countable
sequentially dense subsets of topological space.

\medskip
\begin{center}
$S_1(\mathcal{S},\mathcal{S}) \Rightarrow
S_{fin}(\mathcal{S},\mathcal{S}) \Rightarrow
S_1(\mathcal{S},\mathcal{D}) \Rightarrow
S_{fin}(\mathcal{S},\mathcal{D})$ \\  \, \, $\Uparrow$ \, \, \, \,
\,\, \,  \, \, \, $ \Uparrow $ \,\, \, \, \,\, \, \,
\, \, \, $\Uparrow $ \,\, \, \, \,\, \, \,\, \, \, $\Uparrow$ \\
$S_1(\mathcal{D},\mathcal{S}) \Rightarrow
S_{fin}(\mathcal{D},\mathcal{S}) \Rightarrow
S_1(\mathcal{D},\mathcal{D}) \Rightarrow
S_{fin}(\mathcal{D},\mathcal{D})$

\end{center}

\medskip


 We recall that a subset of $X$ that is the
 complete preimage of zero for a certain function from~$C(X)$ is called a zero-set.
A subset $O\subseteq X$  is called  a cozero-set (or functionally
open) of $X$ if $X\setminus O$ is a zero-set.



\medskip
Recall that the $i$-weight $iw(X)$ of a space $X$ is the smallest
infinite cardinal number $\tau$ such that $X$ can be mapped by a
one-to-one continuous mapping onto a Tychonoff space of the weight
not greater than $\tau$.

\begin{theorem} (Noble's Theorem in \cite{nob}) \label{th31} Let $X$ be a Tychonoff space. A space $C_{p}(X)$ is separable if and only if
$iw(X)=\aleph_0$.
\end{theorem}

\begin{theorem} \label{th30} (Pestrjkov's Theorem in \cite{ps}). Let $X$ be a Tychonoff space
and $0\leq\alpha\leq \omega_1$. A space $B_{\alpha}(X)$ is
separable if and only if $iw(X)=\aleph_0$.
\end{theorem}


\medskip







\begin{definition} A space $X$ has {\bf $O$-property} ($X$ $\models$ $O$), if there
 exist  a  bijection $\varphi: X \mapsto Y$ from a space $X$ onto a
 separable metrizable space $Y$, such that

\begin{enumerate}

\item $\varphi^{-1}(U)$ --- Baire set of $X$ for any open set $U$
of $Y$;

\item  $\varphi(T)$ --- $F_{\sigma}$-set of $Y$ for any Baire set
$T$ of $X$.

\end{enumerate}

\end{definition}

By Corollary 4.4, Corollary 4.5 and Corollary 4.6 in \cite{os1},
we have

\begin{theorem}\label{th37} (Osipov) For each Tychonoff space $X$ the
following are equivalent:

\begin{enumerate}

\item $B(X)$ is sequentially separable;

\item $X$ $\models$ $O$;

\item  $\exists$  a Baire isomorphism $f:X\mapsto M$ of class
$\alpha$ from a space $X$ onto
 a $\sigma$-set $M$ for some $\alpha\in (1,\omega_1)$.

\end{enumerate}

 \end{theorem}

Recall that the cardinal $\mathfrak{p}$ is the smallest cardinal
so that there is a collection of $\mathfrak{p}$ many subsets of
the natural numbers with the strong finite intersection property
but no infinite pseudo-intersection. Note that $\omega_1 \leq
\mathfrak{p} \leq \mathfrak{c}$.

For $f,g\in \mathbb{N}^{\mathbb{N}}$, let $f\leq^{*} g$ if
$f(n)\leq g(n)$ for all but finitely many $n$. $\mathfrak{b}$ is
the minimal cardinality of a $\leq^{*}$-unbounded subset of
$\mathbb{N}^{\mathbb{N}}$. A set $B\subset [\mathbb{N}]^{\infty}$
is unbounded if the set of all increasing enumerations of elements
of $B$ is unbounded in $\mathbb{N}^{\mathbb{N}}$, with respect to
$\leq^{*}$. It follows that $|B|\geq \mathfrak{b}$. A subset $S$
of the real line is called a $Q$-set if each one of its subsets is
a $G_{\delta}$. The cardinal $\mathfrak{q}$ is the smallest
cardinal so that for any $\kappa< \mathfrak{q}$ there is a $Q$-set
of size $\kappa$. (See \cite{do} for more on small cardinals
including $\mathfrak{p}$).

\medskip
For a collection $\mathcal{J}$ of spaces, let $non(\mathcal{J})$
denote the minimal cardinality for a space which is not a member
of $\mathcal{J}$.

\section{$B(X)$ $\models$ $S_1(\mathcal{D},\mathcal{D})$}

\begin{theorem}\label{th24} For a Tychonoff space $X$, the
following are equivalent:

\begin{enumerate}

\item $B(X)$  $\models$ $S_{1}(\Omega_x,\Omega_x)$;

\item $X$ $\models$ $S_{1}(\mathcal{B}_{\Omega},
\mathcal{B}_{\Omega})$.

\end{enumerate}

\end{theorem}

\begin{proof} We prove similarly the proof of Theorem 1 in \cite{sak} for $C_p(X)$.

$(1)\Rightarrow(2)$. Let $\{\mathcal{U}_n\}_{n\in \omega}$ be a
sequence of Baire $\omega$-covers of $X$. We set $A_n=\{f\in B(X):
f\upharpoonright (X\setminus U)=0$ for some $U\in \mathcal{U}_n
\}$. It is not difficult to see that each $A_n$ is dense in $B(X)$
since each $\mathcal{U}_n$ is an $\omega$-cover of $X$ and $X$ is
Tychonoff. Let $f$ be the constant function to 1. By the
assumption there exist $f_n\in A_n$ such that $f\in
\overline{\{f_n: n\in \omega\}}$. For each $f_n$ we take $U_n\in
\mathcal{U}_n$ such that $f_n\upharpoonright(X\setminus U_n)=0$.
Set $\mathcal{U}=\{U_n: n\in \omega\}$. For each finite subset
$\{x_1,...,x_k\}$ of $X$ we consider the basic open neighborhood
of $f$ $[x_1,...,x_k; W,..., W]$, where $W=(0,2)$. Note that
$[x_1,...,x_k; W,..., W]$ contains some $f_n$. This means
$\{x_1,...,x_k\}\subset U_n$. Consequently $\mathcal{U}$ is an
$\omega$-cover of $X$.

$(2)\Rightarrow(1)$. Let $f\in \bigcap\limits_n\overline{A_n}$,
where $A_n$ is a subset of $B(X)$. Since $B(X)$ is homogeneous, we
may think that $f$ is the constant function to the zero. We set
$\mathcal{U}_n=\{g^{-1}(-1/n, 1/n) : g\in A_n\}$ for each $n\in
\omega$. For each $n\in \omega$ and each finite subset $\{x_1,
x_2,..., x_k\}$ of $X$ a neighborhood $[x_1,...,x_k; W,..., W]$ of
$f$, where $W=(-1/n, 1/n)$, contains some $g\in A_n$. This means
that each $\mathcal{U}_n$ is the Baire $\omega$-cover of $X$. In
case the set $M=\{n\in \omega: X\in \mathcal{U}_n \}$ is infinite,
choose $g_{m}\in A_m$ $m\in M$ so that $g^{-1}(-1/m, 1/m)=X$, then
$g_m \mapsto f$. So we may assume that there exists $n\in \omega$
such that for each $m\geq n$ and $g\in A_m$ $g^{-1}(-1/m, 1/m)$ is
not $X$. For the sequence $\{\mathcal{U}_m : m\geq n\}$ of Baire
$\omega$-covers there exist $f_m\in A_m$ such that
$\mathcal{U}=\{f^{-1}_m(-1/m,1/m): m\geq n\}$ is Baire
$\omega$-cover of $X$. Let $[x_1,...,x_k; W,...,W]$ be any basic
open neighborhood of $f$, where $W=(-\epsilon, \epsilon)$,
$\epsilon>0$. There exists $m\geq n$ such that
$\{x_1,...,x_k\}\subset f^{-1}_m(-1/m, 1/m)$ and $1/m<\epsilon$.
This means $f\in \overline{\{f_m: m\geq n\}}$.

\end{proof}

\begin{theorem}\label{th22} For a Tychonoff space $X$, the
following are equivalent:

\begin{enumerate}

\item $B(X)$ $\models$ $S_{1}(\mathcal{D},\mathcal{D})$ and it is
separable;

\item $X$ $\models$ $S_{1}(\mathcal{B}_{\Omega},
\mathcal{B}_{\Omega})$ and $iw(X)=\aleph_0$;

\item $B(X)$ $\models$ $S_{1}(\Omega_x, \Omega_x)$ and it is
separable;

\item $B(X)$ $\models$ $S_{1}(\mathcal{D}, \Omega_x)$ and it is
 separable.

\end{enumerate}

\end{theorem}

\begin{proof} $(1)\Rightarrow(2)$. Let $X$ be a Tychonoff space
satisfying the hypotheses. By Pestrjkov's Theorem \ref{th30},
$iw(X)=\aleph_0$. Consider a condensation (one-to-one continuous
mapping) $g: X\mapsto Z$ from a space $X$ onto a separable
metrizable space $Z$. Let $\beta$ be a countable base of $Z$.
Denote $g^{-1}(\beta):=\{g^{-1}(B): B\in \beta\}$.

 Let $\{\mathcal{B}_i\}_{i\in \omega}$ be a sequence
of countable Borel $\omega$-covers of $X$ where
$\mathcal{B}_i=\{W^{j}_{i}\}_{j\in \omega}$ for each $i\in
\omega$.

Consider a topology $\tau$ generated by the family
$\mathcal{P}=\{W^{j}_{i}\cap A_s: i,j\in \omega$ and $A_s\in
g^{-1}(\beta)\}\bigcup \{(X\setminus W^{j}_{i})\cap A_s: i,j\in
\omega$ and $A_s\in g^{-1}(\beta)\} $.

Note that if $\chi_{P}$ is  a characteristic function of $P$ for
each $P\in \mathcal{P}$, then a diagonal mapping
$\varphi=\Delta_{P\in \mathcal{P}} \chi_{P} : X\mapsto 2^{\omega}$
is a Baire bijection onto $Z=\varphi(X)\subset 2^{\omega}$.
 A space $Z=(X,\tau)$ is a separable metrizable space. Note that
 $\varphi(\mathcal{B}_i)=\{\varphi(B): B\in \mathcal{B}_i \}$ is countable open $\omega$-cover of $Z$
 for each $i\in \omega$. Since $B(\varphi(X))=B(Z)$ is a dense subset of $B(X)$, then
  $B(Z)$ also has property $S_{1}(\mathcal{D},\mathcal{D})$. Since
  $C_p(Z)$ is a dense subset of $B(Z)$, $C_p(Z)$ also has property $S_{1}(\mathcal{D},\mathcal{D})$.

  By  Theorem 3.3 in \cite{os1}, the space $Z$ has property $S_{1}(\Omega,
  \Omega)$. It follows that there is a sequence $\{W^{j(i)}_{i}\}_{i\in
  \omega}$ such that $W^{j(i)}_{i}\in \mathcal{B}_i$ and $\{\varphi(W^{j(i)}_{i}): i\in
  \omega\}$ is open $\omega$-cover of $Z$. It follows that $\{W^{j(i)}_{i}: i\in
  \omega\}$ is Baire $\omega$-cover of $X$.

$(2)\Rightarrow(1)$. Assume that $X$ has property
$S_{1}(B_{\Omega}, B_{\Omega})$ and $iw(X)=\aleph_0$. By
Pestrjkov's Theorem, $B(X)$ is a separable space.

 Let $\{D_i\}_{i\in \omega}$ be a sequence countable dense subsets of
$B(X)$. We claim that for any $f\in B(X)$ there is a sequence
$\{f_k\}\subset B(X)$ such that $f_k\in D_k$ for each $k\in
\omega$ and $f\in \overline{\{f_k : k\in \omega\}}$. Without loss
 of generality we can assume $f=\bf{0}$. Denote $W^{k}_i=\{ x\in X
 : -\frac{1}{k}<f^k_i(x)<\frac{1}{k} \}$ for each $f^k_i\in D_k=\{f^k_i : i\in \omega\}$ and
 $k\in \omega$.

If for each $j\in \omega$ there is $k(j)$ such that
$W^{k(j)}_{i(j)}=X$, then a sequence $f_{k(j)}=f^{k(j)}_{i(j)}$
uniform convergence to $f$ and, hence, $f\in \overline{\{f_{k(j)}
: j\in \omega\}}$.

We can assume that $W^{k}_i\neq X$ for any $k,i\in \omega$.

(a). $\{W^{k}_i\}_{i\in
 \omega}$ a sequence of Baire sets of $X$.

(b). For each $k\in \omega$, $\{W^{k}_i :i\in
 \omega\}$ is a $\omega$-cover of $X$.

By (2),  $X$ has property $S_{1}(\mathcal{B}_{\Omega},
\mathcal{B}_{\Omega})$, hence, there is a sequence
$\{W^{k}_{i(k)}\}_{k\in \omega}$ such that $W^{k}_{i(k)}\in
\{W^{k}_i: i\in
 \omega\}$  for each $k\in \omega$ and $\{W^{k}_{i(k)} : k\in \omega\}
 $ is a $\omega$-cover of $X$.

 Consider $\{f^{k}_{i(k)}\}_{k\in \omega}$ and we claim that $f\in \overline{\{f^{k}_{i(k)} : k\in
 \omega\}}$.
Let $K$ be a finite subset of $X$, $\epsilon>0$ and $U=<f, K,
\epsilon>$ be a base neighborhood of $f$, then there is $k_0\in
\omega$ such that $\frac{1}{k_0}<\epsilon$ and $K\subset
W^{k_0}_{i(k_0)}$. It follows that $f^{k_0}_{i(k_0)}\in U$.

Let $D=\{d_n: n\in \omega \}$ be a dense subspace of $B(X)$. Given
a sequence $\{D_i\}_{i\in \omega}$ of dense subspace of $B(X)$,
enumerate it as $\{D_{n,m}: n,m \in \omega \}$. For each $n\in
\omega$, pick $d_{n,m}\in D_{n,m}$ so that $d_n\in
\overline{\{d_{n,m}: m\in \omega\}}$. Then $\{d_{n,m}: m,n\in
\omega\}$ is dense in $B(X)$.

$(2)\Leftrightarrow(3)$. By Theorem \ref{th24} and Theorem
\ref{th30}.

$(1)\Leftrightarrow(4)$.  By Proposition 52 in \cite{bbm1} for
countable dense subsets of $B(X)$.
\end{proof}

\section{$B(X)$ $\models$ $S_{fin}(\mathcal{D},\mathcal{D})$}

\begin{theorem}\label{th34} For a Tychonoff space $X$, the
following are equivalent:

\begin{enumerate}

\item $B(X)$  $\models$ $S_{fin}(\Omega_x, \Omega_x)$;

\item $X$ $\models$ $S_{fin}(\mathcal{B}_{\Omega},
\mathcal{B}_{\Omega})$.

\end{enumerate}

\end{theorem}

\begin{proof}
$(1)\Rightarrow(2)$. Let $\{\mathcal{U}_n : n\in \omega\}$ be a
sequence of Baire $\omega$-covers of $X$. We set $A_n=\{f\in B(X):
f\upharpoonright (X\setminus U)=0$ for some $U\in \mathcal{U}_n
\}$. It is not difficult to see that each $A_n$ is dense in $B(X)$
since each $\mathcal{U}_n$ is an $\omega$-cover of $X$ and $X$ is
Tychonoff. Let $f$ be the constant function to 1. By the
assumption there exist $F_n\subset A_n$ such that $f\in
\overline{\bigcup\{F_n: n\in \omega\}}$. For each $F_n$ we take
$W_n=\{U^j_n\}_{j=1}^{k(n)}\subset \mathcal{U}_n$ such that
$f^j_n\upharpoonright(X\setminus U^j_n)=0$ for $f^j_n\in F_n$ and
$U^j_n\in W_n$ where $j=1,...,k(n)$.

 Set $\mathcal{U}=\bigcup_{n\in \omega} W_n$. For each finite subset
$\{x_1,...,x_k\}$ of $X$ we consider the basic open neighborhood
of $f$ $[x_1,...,x_k; W,..., W]$, where $W=(0,2)$. Note that
$[x_1,...,x_k; W,..., W]$ contains some $f^{j(n)}_n$. This means
$\{x_1,...,x_k\}\subset U^{j(n)}_n$. Consequently $\mathcal{U}$ is
an $\omega$-cover of $X$.

$(2)\Rightarrow(1)$. Let $f\in \bigcap\limits_n\overline{A_n}$,
where $A_n$ is a subset of $B(X)$. Since $B(X)$ is homogeneous, we
may think that $f$ is the constant function to the zero. We set
$\mathcal{U}_n=\{g^{-1}(-1/n, 1/n) : g\in A_n\}$ for each $n\in
\omega$. For each $n\in \omega$ and each finite subset $\{x_1,...,
x_k\}$ of $X$ a neighborhood $[x_1,...,x_k; W,..., W]$ of $f$,
where $W=(-1/n, 1/n)$, contains some $g\in A_n$. This means that
each $\mathcal{U}_n$ is the Baire $\omega$-cover of $X$. In case
the set $M=\{n\in \omega: X\in \mathcal{U}_n \}$ is infinite,
choose $g_{m}\in A_m$ $m\in M$ so that $g^{-1}(-1/m, 1/m)=X$, then
$g_m \mapsto f$. So we may assume that there exists $n\in \omega$
such that for each $m\geq n$ and $g\in A_m$ $g^{-1}(-1/m, 1/m)$ is
not $X$. For the sequence $\{\mathcal{U}_m : m\geq n\}$ of Baire
$\omega$ there exist $F_m=\{f_{m,1},..., f_{m,k(m)}\}\subset A_m$
such that

$\mathcal{U}=\bigcup_{m\geq n} \{f^{-1}_{m,1}(-1/m,1/m),...,
f^{-1}_{m,k(m)}(-1/m,1/m): i=1,...,k(m) \}$ is a Baire
$\omega$-cover of $X$. Let $[x_1,...,x_k; W,...,W]$ be any basic
open neighborhood of $f$, where $W=(-\epsilon, \epsilon)$,
$\epsilon>0$. There exists $m\geq n$ and $j\in \{1,...,k(m)\}$
such that $\{x_1,...,x_k\}\subset f^{-1}_{m,j}(-1/m, 1/m)$ and
$1/m<\epsilon$. This means $f\in \overline{\bigcup \{F_m: m\geq
n\}}$.

\end{proof}

\begin{theorem}\label{th35} For a Tychonoff space $X$, the
following are equivalent:

\begin{enumerate}

\item $B(X)$ $\models$ $S_{fin}(\mathcal{D},\mathcal{D})$ and it
is separable;

\item $X$ $\models$ $S_{fin}(\mathcal{B}_{\Omega},
\mathcal{B}_{\Omega})$ and $iw(X)=\aleph_0$;

\item $B(X)$ $\models$ $S_{fin}(\Omega_x, \Omega_x)$ and it is
separable;

\item $B(X)$ $\models$ $S_{fin}(\mathcal{D}, \Omega_x)$ and it is
separable.

\end{enumerate}

\end{theorem}

\begin{proof} $(1)\Rightarrow(2)$. Let $X$ be a Tychonoff space
satisfying the hypotheses. By Pestrjkov's Theorem \ref{th30},
$iw(X)=\aleph_0$. Consider a condensation (one-to-one continuous
mapping) $g: X\mapsto Z$ from a space $X$ onto a separable
metrizable space $Z$. Let $\beta$ be a countable base of $Z$.
Denote $g^{-1}(\beta):=\{g^{-1}(B): B\in \beta\}$.

 Let $\{\mathcal{B}_i\}_{i\in \omega}$ be a sequence
of countable Baire $\omega$-covers of $X$ where
$\mathcal{B}_i=\{W^{j}_{i}: j\in \omega\}$ for each $i\in \omega$.

Consider a topology $\tau$ generated by the family
$\mathcal{P}=\{W^{j}_{i}\cap A_s: i,j\in \omega$ and $A_s\in
g^{-1}(\beta)\}\bigcup \{(X\setminus W^{j}_{i})\cap A_s: i,j\in
\omega$ and $A_s\in g^{-1}(\beta)\} $.

 A space $Z=(X,\tau)$ is a separable metrizable space. Note that
 $\{\mathcal{B}_i\}$ is countable open $\omega$-cover of $Z$
 for each $i\in \omega$. Let $\varphi$ be a Baire mapping $\varphi: X
\mapsto Z$ from a space $X$ on a separable metric space $Z$. Then
$B(\varphi(X))=B(Z)$ is a dense subset of $B(X)$ and, hence,
  $B(Z)$ has property $S_{fin}(\mathcal{D},\mathcal{D})$. Since
  $C_p(Z)$ is a dense subset of $B(Z)$, $C_p(Z)$ also has property $S_{fin}(\mathcal{D},\mathcal{D})$.

  By Theorem 21 in \cite{bbm1} and \cite{jmss}, the space $Z$ has property $S_{fin}(\Omega,
  \Omega)$. It follows that there is a sequence $\{F_i=\{W^{j_1}_{i},...,W^{j_s}_{i}\} \}_{i\in
  \omega}$ such that $F_i\subset\mathcal{B}_i$ and
  $\bigcup_{i\in\omega}F_{i}\in \Omega$. It follows that $\bigcup_{i\in\omega}F_{i}$ is Baire $\omega$-cover of $X$.

$(2)\Rightarrow(1)$. Assume that $X$ has property
$S_{fin}(B_{\Omega}, B_{\Omega})$. Let $\{D_i\}_{i\in \omega}$ be
a sequence countable dense subsets of $B(X)$. We claim that for
any $f\in B(X)$ there is a sequence $\{F_k\}$ of finite subsets of
$B(X)$ such that $F_k\subset D_k$ for each $k\in \omega$ and $f\in
\overline{\bigcup\{F_k : k\in \omega\}}$.

Without loss
 of generality we can assume $f=\bf{0}$. Denote $W^{k}_i=\{ x\in X
 : -\frac{1}{k}<f^k_i(x)<\frac{1}{k} \}$ for each $f^k_i\in D_k=\{f^k_i : i\in \omega\}$ and
 $k\in \omega$.

If for each $j\in \omega$ there is $k(j)$ such that
$W^{k(j)}_{i(j)}=X$, then a sequence $f_{k(j)}=f^{k(j)}_{i(j)}$
uniform convergence to $f$ and, hence, $f\in \overline{\{f_{k(j)}
: j\in \omega\}}$.

We can assume that $W^{k}_i\neq X$ for any $k,i\in \omega$.

(a). $\{W^{k}_i\}_{i\in
 \omega}$ a sequence of Baire sets of $X$.

(b). For each $k\in \omega$, $\{W^{k}_i\}_{i\in
 \omega}$ is a $\omega$-cover of $X$.

By (2), $X$ has property $S_{fin}(\mathcal{B}_{\Omega},
\mathcal{B}_{\Omega})$, hence, there is a sequence
$\{S_k=\{W^{k}_{i(1)},...,W^{k}_{i(s(k))}\}\}_{k\in \omega}$ such
that $S_k\subset  \{W^{k}_i\}_{i\in
 \omega}$  for each $k\in \omega$ and $\bigcup_{k\in \omega} S_k$ is a $\omega$-cover of $X$.

 Consider $\{f^{k}_{i(j)} : j=1,...,s(k),k\in \omega\}$ and we claim that

 $f\in \overline{\{f^{k}_{i(j)} :j=1,...,s(k),k\in \omega\}}$.

Let $K$ be a finite subset of $X$, $\epsilon>0$ and $U=<f, K,
\epsilon>$ be a base neighborhood of $f$, then there is $k_0\in
\omega$ such that $\frac{1}{k_0}<\epsilon$ and $K\subset
W^{k_0}_{i(j')}$. It follows that $f^{k_0}_{i(j')}\in U$ for some
$j'\in \{1,...,s(k_0)\}$.

Let $D=\{d_n: n\in \omega \}$ be a dense subspace of $B(X)$. Given
a sequence $\{D_i\}_{i\in \omega}$ of dense subspace of $B(X)$,
enumerate it as $\{D_{n,m}: n,m \in \omega \}$. For each $n\in
\omega$, pick finite set $F_{n,m}\subset D_{n,m}$ so that $d_n\in
\overline{\bigcup_{m\in \omega} F_{n,m}}$. Then $\bigcup_{m,n\in
\omega} F_{n,m}$ is dense in $B(X)$.

$(2)\Leftrightarrow(3)$. By Theorem \ref{th34} and Theorem
\ref{th30}.

$(1)\Leftrightarrow(4)$. By Proposition 15  in \cite{bbm1} for
countable dense subsets of $B(X)$.
\end{proof}

\section{$B(X)$ $\models$ $S_1(\mathcal{D},\mathcal{S})$}

\begin{theorem}\label{th66} For a Tychonoff space $X$, the following statements are
equivalent:

\begin{enumerate}

\item $B(X)$ $\models$  $S_{1}(\Omega_x, \Gamma_x)$;

\item $X$ $\models$ $S_{1}(\mathcal{B}_{\Omega},
\mathcal{B}_{\Gamma})$.

\end{enumerate}

\end{theorem}

\begin{proof}

$(1)\Rightarrow(2)$. Let $\{\mathcal{U}_n : n\in \omega\}$ be a
sequence of Baire $\omega$-covers of $X$. We set $A_n=\{f\in B(X):
f\upharpoonright (X\setminus U)=0$ for some $U\in \mathcal{U}_n
\}$. It is not difficult to see that each $A_n$ is dense in $B(X)$
since each $\mathcal{U}_n$ is an $\omega$-cover of $X$ and $X$ is
Tychonoff. Let $f$ be the constant function to 1. By the
assumption there exist $f_n\in A_n$ such that $\{f_n\}_{n\in
\omega}$ converge to $f$.

 For each $f_n$ we
take $U_n\in \mathcal{U}_n$ such that
$f_n\upharpoonright(X\setminus U_n)=0$.

 Set $\mathcal{U}=\{ U_n : n\in \omega\}$. For each finite subset
$\{x_1,...,x_k\}$ of $X$ we consider the basic open neighborhood
of $f$ $[x_1,...,x_k; W,..., W]$, where $W=(0,2)$.

 Note that there is $n'\in \omega$ such that
$[x_1,...,x_k; W,..., W]$ contains $f_n$ for $n>n'$. This means
$\{x_1,...,x_k\}\subset U_n$ for $n>n'$. Consequently
$\mathcal{U}$ is an $\gamma$-cover of $X$.

$(2)\Rightarrow(1)$. Let $f\in \bigcap\limits_n\overline{A_n}$,
where $A_n$ is a subset of $B(X)$. Since $B(X)$ is homogeneous, we
may think that $f$ is the constant function to the zero. We set
$\mathcal{U}_n=\{g^{-1}(-1/n, 1/n) : g\in A_n\}$ for each $n\in
\omega$. For each $n\in \omega$ and each finite subset $\{x_1,...,
x_k\}$ of $X$ a neighborhood $[x_1,...,x_k; W,..., W]$ of $f$,
where $W=(-1/n, 1/n)$, contains some $g\in A_n$. This means that
each $\mathcal{U}_n$ is Baire $\omega$-cover of $X$. In case the
set $M=\{n\in \omega: X\in \mathcal{U}_n \}$ is infinite, choose
$g_{m}\in A_m$ $m\in M$ so that $g^{-1}(-1/m, 1/m)=X$, then $g_m
\mapsto f$. So we may assume that there exists $n\in \omega$ such
that for each $m\geq n$ and $g\in A_m$ $g^{-1}(-1/m, 1/m)$ is not
$X$. For the sequence $\{\mathcal{U}_m : m\geq n\}$ of Baire
$\omega$-covers there exist $f_m\in A_m$ such that

$\mathcal{U}=\{f^{-1}_m(-1/m,1/m): m>n\}$ is a $\gamma$-cover of
$X$. Let $[x_1,...,x_k; W,...,W]$ be any basic open neighborhood
of $f$, where $W=(-\epsilon, \epsilon)$, $\epsilon>0$. There
exists $m'\geq n$  such that $\{x_1,...,x_k\}\subset
f^{-1}_m(-1/m, 1/m)$ and $1/m<\epsilon$ for each $m>m'$. This
means $\{f_m\}$ converge to $f$.

\end{proof}

\begin{theorem}\label{th212} For a Tychonoff space $X$, the following statements are
equivalent:

\begin{enumerate}

\item $B(X)$ $\models$ $S_{1}(\mathcal{D},\mathcal{S})$ and it is
separable;

\item $B(X)$ is strongly sequentially separable;

\item $X$ $\models$ $S_{1}(\mathcal{B}_{\Omega},
\mathcal{B}_{\Gamma})$ and $iw(X)=\aleph_0$;

\item $B(X)$ $\models$ $S_{1}(\Omega_x, \Gamma_x)$ and it is
separable;

\item $B(X)$ $\models$ $S_{1}(\mathcal{D}, \Gamma_x)$ and it is
separable.

\end{enumerate}
\end{theorem}

\begin{proof}

$(1)\Rightarrow(2)$. Let $D=\{d_i: i\in \omega\}$ be a countable
dense subset of
 $B(X)$. By $S_{1}(\mathcal{D},\mathcal{S})$, for
 sequence $\{D_i : D_i=D$ and $i\in \omega \}$
there is a set $\{d_{i}: i\in\omega\}$ such that for each $i$,
$d_{i}\in D_{i}$, and $\{d_{i}: i\in\omega \}$ is a countable
sequentially dense  subset of $B(X)$. It follows that $D$ is a
countable sequentially dense subset of $B(X)$.

$(2)\Leftrightarrow(3)$. By Theorem 3.4 in \cite{obp}.

$(3)\Leftrightarrow(4)$. By Theorem \ref{th66}.

$(4)\Rightarrow(5)$ is immediate.

$(5)\Rightarrow(2)$. Let $D\in \mathcal{D}$, $f\in B(X)$ and
$\{D_n\}_{n\in \omega}$ such that $D_n=D$ for each $n\in \omega$.
By $(5)$, there is a set $\{f_{n}: n\in\omega\}$ such that for
each $n$, $f_{n}\in D_{n}$, and $\{f_{n}\}_{n\in\omega}$ converge
to $f$. It follows that $D$ is a sequentially dense subset of
$B(X)$.

$(3)\Rightarrow(1)$. Let $\{D_i\}_{i\in \omega}$ be a sequence of
countable dense subsets of $B(X)$. Since $X$ $\models$
$S_{1}(\mathcal{B}_{\Omega},
 \mathcal{B}_{\Gamma})$, then,  $X$ $\models$ $S_{1}(\mathcal{B}_{\Omega}, \mathcal{B}_{\Omega})$ and, by
 Theorem \ref{th22}, $B(X)$ $\models$ $S_{1}(\mathcal{D},\mathcal{D})$. Then there is a sequence $\{d_{i}\}_{i\in\omega}$ such that for each $i$,
$d_{i}\in D_{i}$, and $\{d_{i}: i\in\omega \}$ is a countable
dense  subset of $B(X)$. By $(2)$, $\{d_{i}: i\in\omega \}$ is a
countable sequentially dense subset of $B(X)$, i.e. $\{d_{i}:
i\in\omega \}\in \mathcal{S}$.

\end{proof}

Note that $S_{1}(\mathcal{B}_{\Omega},
\mathcal{B}_{\Gamma})=S_{fin}(\mathcal{B}_{\Omega},
\mathcal{B}_{\Gamma})$ (see \cite{scts}).

\begin{theorem}\label{th55} For a Tychonoff space $X$, the following statements are
equivalent:

\begin{enumerate}

\item $B(X)$ $\models$ $S_{fin}(\Omega_x, \Gamma_x)$;

\item $X$ $\models$ $S_{fin}(\mathcal{B}_{\Omega},
\mathcal{B}_{\Gamma})$.

\end{enumerate}

\end{theorem}

\begin{proof} By Theorem \ref{th66}, it suffices to prove $(1)\Rightarrow(2)$.

$(1)\Rightarrow(2)$. Let $\{\mathcal{U}_n\}_{n\in \omega}$ be a
sequence of Baire $\omega$-covers of $X$. We set $A_n=\{f\in B(X):
f\upharpoonright (X\setminus U)=0$ for some $U\in \mathcal{U}_n
\}$. It is not difficult to see that each $A_n$ is dense in $B(X)$
since each $\mathcal{U}_n$ is an $\omega$-cover of $X$ and $X$ is
Tychonoff. Let $f$ be the constant function to 1. By the
assumption there exist $\{f^i_n\}_{i=1}^{k(n)}\subset A_n$ such
that $\bigcup_{n\in \omega} \{f^i_n\}_{i=1}^{k(n)}$ converge to
$f$. Note that a subsequence $\{f^1_n\}_{n\in \omega}$ also
converge to $f$.

 For each $f^1_n$ we
take $U_n\in \mathcal{U}_n$ such that
$f^1_n\upharpoonright(X\setminus U_n)=0$.

 Set $\mathcal{U}=\{ U_n : n\in \omega\}$. For each finite subset
$\{x_1,...,x_k\}$ of $X$ we consider the basic open neighborhood
of $f$ $[x_1,...,x_k; W,..., W]$, where $W=(0,2)$.

 Note that there is $n'\in \omega$ such that
$[x_1,...,x_k; W,..., W]$ contains $f^1_n$ for $n>n'$. This means
$\{x_1,...,x_k\}\subset U_n$ for $n>n'$. Consequently
$\mathcal{U}$ is an $\gamma$-cover of $X$.
\end{proof}

\begin{theorem}\label{th123} For a Tychonoff space $X$, the following statements are
equivalent:

\begin{enumerate}

\item $B(X)$ $\models$ $S_{fin}(\mathcal{D},\mathcal{S})$ and it
is separable;

\item $B(X)$ is strongly sequentially separable;

\item $X$ $\models$ $S_{fin}(\mathcal{B}_{\Omega},
\mathcal{B}_{\Gamma})$ and $iw(X)=\aleph_0$;

\item $B(X)$ $\models$ $S_{fin}(\Omega_x, \Gamma_x)$ and it is
separable;

\item $B(X)$ $\models$ $S_{fin}(\mathcal{D}, \Gamma_x)$ and it is
separable.

\end{enumerate}

\end{theorem}

\begin{proof} By Theorem \ref{th55} and Theorem \ref{th212}.

\end{proof}

\section{$B(X)$ $\models$ $S_1(\mathcal{S},\mathcal{D})$}

Recall that $U_{fin}(\mathcal{S},\mathcal{D})$ is the selection
hypothesis:
 whenever $\mathcal{U}_1$, $\mathcal{U}_2, ... \in \mathcal{S}$, there are finite sets $\mathcal{F}_n\subseteq \mathcal{U}_n$,
$n\in \omega$, such that $\{\bigcup \mathcal{F}_n : n\in
\omega\}\in \mathcal{D}$. For a function space $B(X)$, we can
represent the condition  $\{\bigcup \mathcal{F}_n : n\in
\omega\}\in \mathcal{D}$ as $\forall$ $f\in B(X)$ $\forall$ a base
neighborhood $O(f)=<f, K, \epsilon >$ of $f$  where $\epsilon>0$
and $K=\{x_1, ..., x_k\}$ is a finite subset of $X$, there is
$n'\in \omega$ such that for each $j\in \{1,...,k\}$ there is
$g\in \mathcal{F}_{n'}$ such that $g(x_j)\in (f(x_j)-\epsilon,
f(x_j)+\epsilon)$.

Similarly, $U_{fin}(\Gamma_0,\Omega_0)$: whenever $\mathcal{S}_1$,
$\mathcal{S}_2, ... \in \Gamma_0$, there are finite sets
$\mathcal{F}_n\subseteq \mathcal{S}_n$, $n\in \omega$, such that
$\{\bigcup \mathcal{F}_n : n\in \omega\}\in \Omega_0$, i.e. for a
base neighborhood $O(f)=<f, K, \epsilon
>$ of $f={\bf 0}$  where $\epsilon>0$
and $K=\{x_1, ..., x_k\}$ is a finite subset of $X$, there is
$n'\in \omega$ such that for each $j\in \{1,...,k\}$ there is
$g\in \mathcal{F}_{n'}$ such that $g(x_j)\in (f(x_j)-\epsilon,
f(x_j)+\epsilon)$.

\begin{theorem}\label{th333} For a Tychonoff space $X$, the following statements are
equivalent:

\begin{enumerate}

\item $B(X)$ $\models$ $S_{1}(\Gamma_x,\Omega_x)$;

\item $X$ $\models$ $S_{1}(\mathcal{B}_{\Gamma},
\mathcal{B}_{\Omega})$.

\item $B(X)$ $\models$ $S_{fin}(\Gamma_x,\Omega_x)$;

\item $B(X)$ $\models$ $U_{fin}(\Gamma_x,\Omega_x)$;

\end{enumerate}

\end{theorem}

\begin{proof}

$(1)\Rightarrow(2)$. Let $B(X)$ $\models$
$S_{1}(\Gamma_x,\Omega_x)$ and $\{\mathcal{F}_i : i\in
\omega\}\subset \mathcal{B}_{\Gamma}$. Let
$\mathcal{F}_i=\{F_{i,k} : k\in \omega\}$ for each $i\in \omega$
and $f=\bf{0}$. For each $i\in \omega$ we consider a sequence
$\{f_{i,k}\}_{k\in \omega}$ such that $f_{i,k}\upharpoonright
F_{i,k}=0$ and $f_{i,k}\upharpoonright (X\setminus F_{i,k})=1$ for
each $k\in \omega$.

Since $\mathcal{F}_i$ is a Baire $\gamma$-cover of $X$, we have
that $\{f_{i,k}\}_{k\in \omega}$ converge to $f$ for each $i\in
\omega$.

Since $B(X)$ $\models$ $S_{1}(\Gamma_x,\Omega_x)$, there is a
sequence $\{f_{i,k(i)}\}_{i\in\omega}$ such that for each $i$,
$f_{i,k(i)}\in \{f_{i,k} : k\in \omega\}$, and $\{f_{i,k(i)}:
i\in\omega \}$ is an element of $\Omega_{f}$.

Consider a sequence $\{F_{i,k(i)}: i\in \omega\}$.

(a). $F_{i,k(i)}\in \mathcal{F}_{i}$.

(b). $\{F_{i,k(i)}: i\in \omega\}$ is a $\omega$-cover of $X$.

Let $K$ be a finite subset of $X$ and $U=<f, K, \frac{1}{2}>$ be a
base neighborhood of $f$, then there is $f_{i_{j_0},k(i_{j_0})}\in
U$. It follows that $K\subset F_{i_{j_0},k(i_{j_0})}$. We thus get
$X$ $\models$ $S_{1}(B_{\Gamma}, B_{\Omega})$.

$(2)\Rightarrow(1)$. Let $\{f_{k,i}\}_{k\in \omega}$ be a sequence
converge to $f$ for each $i\in \omega$. Without loss of generality
we can assume that $f=\bf{0}$ and a set $W^{i}_{k}=\{x\in X:
-\frac{1}{i}< f_{k,i}(x)< \frac{1}{i}\}\neq X$ for any $i\in
\omega$.

Consider $\mathcal{V}_i=\{W^{i}_{k} : k\in \omega\}$ for each
$i\in \omega$.

(a). $\mathcal{V}_i$ is Baire cover of $X$.

(b). $\mathcal{V}_i$ is $\gamma$-cover of $X$. Since
$\{f_{k,i}\}_{k\in \omega}$ converge to $f$, for each finite set
$K\subset X$ there is $k_0\in \omega$ such that  $f_{k,i}\in <f,
K, \frac{1}{i}
>$ for $k>k_0$. It follows that $K\subset W^{i}_{k}$ for any
$k>k_0$.

By $X$ $\models$ $S_{1}(\mathcal{B}_{\Gamma},
\mathcal{B}_{\Omega})$, there is a sequence $\{W^{i}_{k(i)}\}_{
i\in\omega}$ such that for each $i$, $W^{i}_{k(i)}\in
\mathcal{V}_i$, and $\{W^{i}_{k(i)}: i\in\omega \}$ is an element
of $\mathcal{B}_{\Omega}$.

We claim that $f\in \overline{\{f_{k(i),i} : i\in \omega \}}$. Let
$U=<f, K, \epsilon>$ be a base neighborhood of $f$ where
$\epsilon>0$ and $K$ is a finite subset of $X$, then there are
$i_0, i_1\in \omega$ such that $\frac{1}{i_0}<\epsilon$, $i_1>i_0$
and $W^{i_1}_{k(i_1)}\supset K$. It follows that
$f_{k(i_1),i_1}\in <f, K, \epsilon>$ and, hence, $f\in
\overline{\{f_{k(i),i} : i\in \omega \}}$.

$(1)\Rightarrow(3)\Rightarrow(4)$ is immediate.

 $(4)\Rightarrow(2)$. Let $\{\mathcal{U}_i: i\in \omega\}\subset \mathcal{B}_{\Gamma}$. Let $\mathcal{U}_i=\{U^m_i : m\in \omega\}$
 we consider $\mathcal{K}_i =\{ f^m_i\in
B(X) : f^m_i\upharpoonright U^{m}_i=0$ and $f^m_i\upharpoonright
(X\setminus U^{m}_i)=1$ for $m \in \omega \}$. Since
$\mathcal{U}_i$ is a $\gamma$-cover of $X$, we have that
$\mathcal{K}_i$ converge to $\bf{0}$ for each $i\in \omega$. By
$B(X)$ $\models$ $U_{fin}(\Gamma_x,\Omega_x)$, there are finite
sets $F_i=\{f^{m_1}_i, ..., f^{m_{s(i)}}_i\}\subseteq
\mathcal{K}_i$ such that $\{\bigcup F_i : i\in \omega\}\in
\Omega_{0}$. Note that $\{\bigcup \{U^{m_1}_i, ...,
U^{m_{s(i)}}_i\}  : i\in \omega\}\in B_{\Omega}$. It follows that
$X$ $\models$
$U_{fin}(\mathcal{B}_{\Gamma},\mathcal{B}_{\Omega})$. By Theorem 9
in \cite{scts},
$S_1(\mathcal{B}_{\Gamma},\mathcal{B}_{\Omega})\Leftrightarrow
S_{fin}(\mathcal{B}_{\Gamma},\mathcal{B}_{\Omega})\Leftrightarrow
U_{fin}(\mathcal{B}_{\Gamma},\mathcal{B}_{\Omega})$.

\end{proof}

\begin{theorem}\label{th111} For a Tychonoff space $X$, the following statements are
equivalent:

\begin{enumerate}

\item $B(X)$ $\models$ $S_{1}(\mathcal{S},\mathcal{D})$ and it is
sequentially separable;

\item $X$ $\models$ $S_{1}(\mathcal{B}_{\Gamma},
\mathcal{B}_{\Omega})$ and $X$ $\models$ $O$;

\item $B(X)$ $\models$ $S_{1}(\Gamma_x,\Omega_x)$ and it is
sequentially separable;

\item $B(X)$ $\models$ $S_{1}(\mathcal{S},\Omega_x)$ and it is
sequentially separable;

\item $B(X)$ $\models$ $S_{fin}(\mathcal{S},\mathcal{D})$ and it
is sequentially separable;

\item $B(X)$ $\models$ $U_{fin}(\mathcal{S},\mathcal{D})$ and it
is sequentially separable;

\item $B(X)$ $\models$ $S_{fin}(\mathcal{S},\Omega_x)$ and it is
sequentially separable;

\item $B(X)$ $\models$ $U_{fin}(\mathcal{S},\Omega_x)$ and it is
sequentially separable.
\end{enumerate}

\end{theorem}

\begin{proof} $(1)\Rightarrow(5)\Rightarrow(6)$ is immediate.

$(6)\Rightarrow(2)$. Let $\{\mathcal{F}_i : i\in \omega\}\subset
\mathcal{B}_{\Gamma}$ where $\mathcal{F}_i=\{ F^{m}_i : m\in
\omega\}$  and $\mathcal{S}=\{h_m : m\in \omega\}$ be a countable
sequentially dense subset of $B(X)$. For each $i\in \omega$ we
consider a countable sequentially dense subset $\mathcal{S}_i$ of
$B(X)$. where $\mathcal{S}_i :=\{ f^m_i\in B(X) :
f^m_i\upharpoonright F^{m}_i=h_m$ and $f^m_i\upharpoonright
(X\setminus F^{m}_i)=1$ for $m \in \omega \}$.

Since $\mathcal{F}_i$ is a Baire $\gamma$-cover of $X$ and
$\mathcal{S}$ is a countable sequentially dense subset of $B(X)$,
we have that $\mathcal{S}_i$ is a  countable sequentially dense
subset of $B(X)$ for each $i\in \omega$.  Let $h\in B(X)$, there
is a sequence $\{h_{m_s}\}_{s\in \omega}$ such that $\{h_{m_s} :
s\in \omega\}\subset \mathcal{S}$ and $\{h_{m_s}\}_{s\in \omega}$
converge to $h$.
 Let $K$ be a finite subset of $X$, $\epsilon>0$ and $W=<h, K,\epsilon>$ be
a base neighborhood of $h$, then there is a number $m_0$ such that
$K\subset F^{m}_i $ for $m>m_0$ and $h_{m_s}\in W$ for $m_s>m_0$.
Since $f^{m_s}_i\upharpoonright K= h_{m_s}\upharpoonright K$ for
each $m_s>m_0$, $f^{m_s}_i\in W$ for each $m_s>m_0$. It follows
that a sequence $\{f^{m_s}_i\}_{s\in \omega}$ converge to $h$.

By $B(X)$ $\models$ $U_{fin}(\mathcal{S},\mathcal{D})$,  there are
finite sets $F_i=\{f^{m_1}_i, ..., f^{m_{s(i)}}_i\}\subseteq
\mathcal{S}_i$ such that $\{\bigcup F_i : i\in \omega\}\in
\mathcal{D}$.

Consider a sequence $\{F^{m_1}_{i}, ..., F^{m_{s(i)}}_{i} : i\in
\omega\}$.

(a). $F^{m_k}_{i}\in \mathcal{F}_{i}$ for $k\in
\overline{1,s(i)}$.

(b). $\{\bigcup\limits_{k=1}^{s(i)} F^{m_k}_{i}: i\in \omega\}$ is
a $\omega$-cover of $X$.

Let $K$ be a finite subset of $X$ and $U=<$ $\bf{0}$ $, K,
\frac{1}{2}>$ be a base neighborhood of $\bf{0}$, then there is
$F_i$ such that $K\subset \bigcup\limits_{k=1}^{s(i)}
F^{m_k}_{i}$. We thus get $X$ $\models$ $U_{fin}(B_{\Gamma},
B_{\Omega})$. By Theorem 9 in \cite{scts},
 $S_1(\mathcal{B}_{\Gamma},\mathcal{B}_{\Omega})\Leftrightarrow
S_{fin}(\mathcal{B}_{\Gamma},\mathcal{B}_{\Omega})\Leftrightarrow
U_{fin}(\mathcal{B}_{\Gamma},\mathcal{B}_{\Omega})$.

$(2)\Rightarrow(3)$. Let $X$ $\models$ $S_{1}(B_{\Gamma},
B_{\Omega})$ and $\{f_{i,m}\}_{m\in \omega}$ converge to $\bf{0}$
for each $i\in \omega$.

Consider $\mathcal{F}_i=\{F_{i,m} : m\in
\omega\}=\{f^{-1}_{i,m}(-\frac{1}{i}, \frac{1}{i}): m\in \omega
\}$ for each $i\in \omega$. Without loss of generality we can
assume that a set $F_{i,m}\neq X$ for any $i,m\in \omega$.
Otherwise there is sequence $\{f_{i_k,m_k}\}_{k\in \omega}$ such
that $\{f_{i_k,m_k}\}_{k\in \omega}$ uniform converge to $\bf{0}$
and $\{f_{i_k,m_k}: k\in \omega\}\in \Omega_{\bf 0}$.

Note that $\mathcal{F}_i$ is Baire $\gamma$-cover of $X$ for each
$i\in \omega$. By $X$ $\models$ $S_{1}(\mathcal{B}_{\Gamma},
\mathcal{B}_{\Omega})$, there is a sequence $\{F_{i,m(i)}:
i\in\omega\}$ such that for each $i$, $F_{i,m(i)}\in
\mathcal{F}_i$, and $\{F_{i,m(i)}: i\in\omega \}$ is an element of
$\mathcal{B}_{\Omega}$.

We claim that $\bf 0$ $\in \overline{\{f_{i,m(i)} : i\in \omega
\}}$. Let $W=<$ $\bf 0$ $, K, \epsilon>$ be a base neighborhood of
$\bf 0$ where $\epsilon>0$ and $K$ is a finite subset of $X$, then
there are $i_0, i_1\in \omega$ such that $\frac{1}{i_0}<\epsilon$,
$i_1>i_0$ and $F_{i_1,m(i_1)}\supset K$. It follows that $f_{i_1,
m(i_1)}\in <$ $\bf 0$ $, K, \epsilon>$ and, hence, $\bf 0$ $\in
\overline{\{f_{i,m(i)} : i\in \omega \}}$ and $B(X)$ $\models$
$S_{1}(\Gamma_x,\Omega_x)$.

$(3)\Rightarrow(4)$ is immediate.

$(4)\Rightarrow(1)$. Suppose that $B(X)$ is a sequentially
separable and  $B(X)$ $\models$ $S_{1}(\mathcal{S},\Omega_x)$. Let
$D=\{d_n: n\in \omega \}$ be a dense subspace of $B(X)$. Given a
sequence of sequentially dense subspaces of $B(X)$, enumerate it
as $\{S_{n,m}: n,m \in \omega \}$. For each $n\in \omega$, pick
$d_{n,m}\in S_{n,m}$ so that $d_n\in \overline{\{d_{n,m}: m\in
\omega\}}$. Then $\{d_{n,m}: m,n\in \omega\}$ is dense in $B(X)$.

$(4)\Rightarrow(7)\Rightarrow(8)$ is immediate.

$(8)\Rightarrow(6)$. Similarly to the implication
$(4)\Rightarrow(1)$.

\end{proof}

\section{$B(X)$ $\models$ $S_1(\mathcal{S},\mathcal{S})$}

Recall that $U_{fin}(\mathcal{S},\mathcal{S})$ is the selection
hypothesis:
 whenever $\mathcal{U}_1$, $\mathcal{U}_2, ... \in \mathcal{S}$, there are finite sets $\mathcal{F}_n\subseteq \mathcal{U}_n$,
$n\in \omega$, such that $\{\bigcup \mathcal{F}_n : n\in
\omega\}\in \mathcal{S}$. For a function space $B(X)$, we can
represent the condition  $\{\bigcup \mathcal{F}_n : n\in
\omega\}\in \mathcal{S}$ as $\forall$ $f\in B(X)$ $\forall$ a base
neighborhood $O(f)=<f, K, \epsilon >$ of $f$  where $\epsilon>0$
and $K=\{x_1, ..., x_k\}$ is a finite subset of $X$, there is
$n'\in \omega$ such that for each $n>n'$ and $j\in \{1,...,k\}$
there is $g\in \mathcal{F}_{n}$ such that $g(x_j)\in
(f(x_j)-\epsilon, f(x_j)+\epsilon)$.

Similarly, $U_{fin}(\Gamma_0,\Gamma_0)$: whenever $\mathcal{S}_1$,
$\mathcal{S}_2, ... \in \Gamma_0$, there are finite sets
$\mathcal{F}_n\subseteq \mathcal{S}_n$, $n\in \omega$, such that
$\{\bigcup \mathcal{F}_n : n\in \omega\}\in \Gamma_0$, i.e. for a
base neighborhood $O(f)=<f, K, \epsilon
>$ of $f={\bf 0}$  where $\epsilon>0$
and $K=\{x_1, ..., x_k\}$ is a finite subset of $X$, there is
$n'\in \omega$ such that for each $n>n'$ and $j\in \{1,...,k\}$
there is $g\in \mathcal{F}_{n}$ such that $g(x_j)\in
(f(x_j)-\epsilon, f(x_j)+\epsilon)$.

\begin{theorem}\label{th445} For a Tychonoff space $X$, the following statements are
equivalent:

\begin{enumerate}

\item  $B(X)$ $\models$ $S_{1}(\Gamma_x, \Gamma_x)$;

\item $X$ $\models$ $S_{1}(\mathcal{B}_{\Gamma},
\mathcal{B}_{\Gamma})$;

\item  $B(X)$ $\models$ $S_{fin}(\Gamma_x, \Gamma_x)$;

\item  $B(X)$ $\models$ $U_{fin}(\Gamma_x, \Gamma_x)$.

\end{enumerate}

\end{theorem}

\begin{proof}

$(1)\Rightarrow(3)\Rightarrow(4)$ is immediate.

$(4)\Rightarrow(2)$. Let $\{\mathcal{U}_i : i\in \omega\}\subset
\mathcal{B}_{\Gamma}$. For each $i\in \omega$ we consider a
 subset $\mathcal{S}_i$ of $B(X)$ and $\mathcal{U}_i=\{
U^{m}_i : m\in \omega\}$ where

$\mathcal{S}_i:=\{ f^m_i\in B(X) : f^m_i\upharpoonright U^{m}_i=0$
and $f^m_i\upharpoonright (X\setminus U^{m}_i)=1$ for $m \in
\omega \}$.

Since $\mathcal{U}_i$ is a $\gamma$-cover of $X$, we have that
$\mathcal{S}_i$ converge to ${\bf 0}$, i.e. $\mathcal{S}_i\in
\Gamma_0$ for each $i\in \omega$.

Since  $B(X)$ $\models$ $U_{fin}(\Gamma_x,\Gamma_x)$, there is a
sequence $\{\mathcal{F}_i\}_{i\in \omega}=\{f^{m_1}_{i},...,
f^{m_{k(i)}}_{i} : i\in\omega\}$ such that for each $i$,
$\mathcal{F}_i\subseteq \mathcal{S}_i$, and $\{\bigcup
\mathcal{F}_i : i\in \omega\}\in \Gamma_0$.

Consider a sequence $\{W_i\}_{i\in \omega}=\{U^{m_1}_{i},
...,U^{m_{k(i)}}_{i} : i\in \omega\}$.

(a). $W_i \subset \mathcal{U}_{i}$.

(b). $\{\bigcup W_i: i\in \omega\}$ is a $\gamma$-cover of $X$.

 Let $K=\{x_1,...,x_s\}$ be a finite subset of $X$ and $U=<$ $\bf{0}$ $,
K, \frac{1}{2}>$ be a base neighborhood of $\bf{0}$, then there
exists $i_0\in \omega$ such that for each $i>i_0$ and

$j\in \{1,...,s\}$ there is $g\in \mathcal{F}_{i}$ such that
$g(x_j)\in (-\frac{1}{2}, \frac{1}{2})$.

 It follows that
$K\subset \bigcup\limits_{j=1}^{k(i)} U^{m_j}_{i}$ for $i>i_0$. We
thus get $X$ $\models$ $U_{fin}(\mathcal{B}_{\Gamma},
\mathcal{B}_{\Gamma})$.

By Theorem 1 in \cite{scts},
$S_1(\mathcal{B}_{\Gamma},\mathcal{B}_{\Gamma})\Leftrightarrow
S_{fin}(\mathcal{B}_{\Gamma},\mathcal{B}_{\Gamma})\Leftrightarrow
U_{fin}(\mathcal{B}_{\Gamma},\mathcal{B}_{\Gamma})$.

$(2)\Rightarrow(1)$. For each $n\in \omega$, let
$\{f_{n,m}\}_{m\in \omega}$ be a sequence of Baire functions on
$X$ converging pointwise to $\bf{0}$. For $n,m\in \omega$, let
$U_{n,m}=\{x\in X: |f_{n,m}(x)|<\frac{1}{2^{n}}\}$.

For each $n\in \omega$, we put $\mathcal{U}_n=\{U_{n,m} : m\in
\omega\}$. If the set $\{n\in \omega: X\in \mathcal{U}_n \}$ is
infinite, $X=U_{n_0,m_0}=U_{n_1,m_1}=... $ for some sequences
$\{n_j\}_{j\in \omega}$ and $\{m_j\}_{j\in \omega}$, where
$\{n_j\}_{j\in \omega}$ is strictly increasing. This means that
$\{f_{n_j,m_j}\}_{j\in \omega}$ is a sequence converging uniformly
to $\bf{0}$. If the set $\{n\in \omega: X\in \mathcal{U}_n \}$ is
finite, by removing such finitely many $n$'s we assume
$U_{n,m}\neq X$ for $n,m\in \omega$. Note that each
$\mathcal{U}_n$ is a $\gamma$-cover of $X$. By $X$ $\models$
$S_{1}(B_{\Gamma}, B_{\Gamma})$, there is a sequence
$\{U_{n,m(n)}\}$ such that $U_{n,m(n)}\in \mathcal{U}_n$ for each
$n\in \omega$ and $\{U_{n,m(n)} : n\in \omega\}$ is a
$\gamma$-cover of $X$. Let $F$ be a finite subset of $X$ and let
$\epsilon$ a positive real number. Because of $U_{n,m(n)}\neq X$
for $n\in \omega$, there is $n'\in \omega$ such that $F\subset
U_{n,m(n)}$ for each $n>n'$ and $\frac{1}{2^{n'}}<\epsilon$. Then
$|f_{n,m(n)}(x)|<\epsilon$ for any $x\in F$ and $n>n'$. Hence
$\{f_{n,m(n)}\}$ converge to $\bf{0}$. Thus $B(X)$ $\models$
$S_{1}(\Gamma_x, \Gamma_x)$.

\end{proof}

\begin{theorem}\label{th444} For a Tychonoff space $X$, the following statements are
equivalent:

\begin{enumerate}

\item $B(X)$ $\models$ $S_{1}(\mathcal{S},\mathcal{S})$ and it is
sequentially separable;

\item $X$ $\models$ $S_{1}(\mathcal{B}_{\Gamma},
\mathcal{B}_{\Gamma})$ and $X$ $\models$ $O$;

\item $B(X)$ $\models$ $S_{1}(\Gamma_x,\Gamma_x)$ and it is
sequentially separable;

\item $B(X)$ $\models$ $S_{1}(\mathcal{S},\Gamma_x)$ and it is
sequentially separable;

\item $B(X)$ $\models$ $S_{fin}(\mathcal{S},\mathcal{S})$ and it
is sequentially separable;

\item $B(X)$ $\models$ $U_{fin}(\mathcal{S},\mathcal{S})$ and it
is sequentially separable;

\item $B(X)$ $\models$ $S_{fin}(\mathcal{S},\Gamma_x)$ and it is
sequentially separable;

\item $B(X)$ $\models$ $U_{fin}(\mathcal{S},\Gamma_x)$ and it is
sequentially separable.

\end{enumerate}

\end{theorem}

\begin{proof}

$(1)\Rightarrow(2)$.  Let $\{\mathcal{F}_i: i\in \omega \}\subset
\mathcal{B}_{\Gamma}$ where $\mathcal{F}_i=\{ F^{m}_i : m\in
\omega\}$ and $\mathcal{S}=\{h_k : k\in \omega\}$ be a countable
sequentially dense subset of $B(X)$. For each $i\in \omega$ we
consider a countable sequentially dense subset $\mathcal{S}_i$ of
$B(X)$ where $\mathcal{S}_i:=\{ f\in B(X) : f\upharpoonright
F^{m}_i=h_m$ and $f\upharpoonright (X\setminus F^{m}_i)=1$ for $m
\in \omega \}$.

Since $\mathcal{F}_i$ is a Baire $\gamma$-cover of $X$ and
$\mathcal{S}$ is a countable sequentially dense subset of $B(X)$,
we have that $\mathcal{S}_i$ is a countable sequentially dense
subset of $B(X)$ for each $i\in \omega$.

Let $h\in B(X)$, there is a sequence $\{h_s\}_{s\in \omega}\subset
\mathcal{S}$ such that $\{h_s\}_{s\in \omega}$ converge to $h$.
 Let $K$ be a finite subset of $X$ and $W=<h, K,\epsilon>$ be a base
neighborhood of $h$, then there is a number $m_0$ such that
$K\subset F^{m}_i $ for $m>m_0$ and $h_s\in W$ for $s>m_0$. Since
$f^s_i\upharpoonright K= h_s\upharpoonright K$ for each $s>m_0$,
$f^s_i\in W$ for each $s>m_0$. It follows that a sequence
$\{f^s_i\}_{s\in \omega}$ converge to $h$.

 Since $B(X)$ $\models$ $S_{1}(\mathcal{S},\mathcal{S})$, there is a set $\{f^{m(i)}_{i}: i\in\omega\}$ such
that for each $i$, $f^{m(i)}_{i}\in \mathcal{S}_i$, and
$\{f^{m(i)}_{i}: i\in\omega \}$ is an element of $\mathcal{S}$.

For $\bf{0}$ $\in B(X)$ there is a set $\{f^{m(i)_j}_{i_j}\}_{j\in
\omega}\subset \{f^{m(i)}_{i}: i\in\omega\}$ such that
$\{f^{m(i)_j}_{i_j}\}_{j\in \omega}$ converge to $\bf{0}$.
Consider a set $\{F^{m(i)_j}_{i_j}: j\in \omega\}$.

(1) $F^{m(i)_j}_{i_j}\in \mathcal{F}_{i_j}$.

(2) $\{F^{m(i)_j}_{i_j}: j\in \omega\}$ is a $\gamma$-cover of
$X$.

Let $K$ be a finite subset of $X$ and $U=<$ $\bf{0}$ $, K,
\frac{1}{2}>$ be a base neighborhood of $\bf{0}$, then there is a
number $j_0$ such that $f^{m(i)_j}_{i_j}\in U$ for any $j>j_0$. It
follows that $K\subset F^{m(i)_j}_{i_j}$ for any $j>j_0$. We thus
get $X\in U_{fin}(\mathcal{B}_{\Gamma}, \mathcal{B}_{\Gamma})$.
Note that $U_{fin}(\mathcal{B}_{\Gamma},
\mathcal{B}_{\Gamma})=S_1(\mathcal{B}_{\Gamma},
\mathcal{B}_{\Gamma})=S_{fin}(\mathcal{B}_{\Gamma},
\mathcal{B}_{\Gamma})$ (see Theorem 1 in \cite{scts}).

 $(2)\Rightarrow(1)$.  Let
$\{S_i\}\subset \mathcal{S}$ and $S\in \mathcal{S}$. Consider a
topology $\tau$ generated by the family
 $\mathcal{P}=\{f^{-1}(G): G$ is an open set of $\mathbb{R}$ and
 $f\in S\cup \bigcup\limits_{i\in \omega} S_i \}$. A space $Y=(X,\tau)$ is a separable metrizable space
 because $P=S\cup \bigcup\limits_{i\in \omega} S_i$ is a countable dense subset
 of $B(X)$. Note that a function $f\in P$, considered as map from
 $Y$ to $\mathbb{R}$, is a continuous function. Note also that the identity
 map $\varphi$ from $X$ on $Y$, is a Borel bijection. By Corollary
 12 in \cite{busu}, $Y$ is a $QN$-space and, hence, by Corollary
 20 in \cite{tszd}, $Y$ has a property $S_{1}(\mathcal{B}_{\Gamma},
 \mathcal{B}_{\Gamma})$. By Corollary 21 in \cite{tszd}, $B(Y)$ is an $\alpha_2$ space and, hence, $B(Y)$ has property
$S_1(\Gamma_z, \Gamma_z)$ for each $z\in B(Y)$.

We claim that there is a sequence $B=\{b_{i}: i\in\omega\}$ such
that for each $i$, $b_{i}\in S_{i}\subset C(Y)$, and, $S\subset
[B]_{seq}$, i.e. for each $d\in S\subset C(Y)$ there is a sequence
$\{b_{i_k} : k\in \omega \}\subset B$ such that $\{b_{i_k} : k\in
\omega \}$ converges to $d$.

Let $q:\omega \mapsto \omega\times \omega$ be a bijection. Then we
enumerate $\{S_i\}$ as $\{S_{q(i)}\}$. For each $d_n\in S$ there
are sequences $s_{n,m}\subset S_{n,m}$ such that $s_{n,m}$
converges to $d_n$ for each $m\in \omega$. By property
$S_1(\Gamma_z, \Gamma_z)$ of $B(Y)$, there is a sequence
$\{b_{n,m}: m\in\omega\}$ such that for each $m$, $b_{n,m}\in
s_{n,m}$, and, $\{b_{n,m}\}$ converges to $d_n$. Denote
$B=\{b_{n,m} : n,m\in \omega \}$. By construction of $B$, we have
that $S\subset [B]_{seq}$.

Note that $B_1(X)=B(X)$ and $\varphi(B(Y))=\varphi(B_1(Y))=B(X)$.

We claim that $B\in \mathcal{S}$. Let $f\in B(Y)$ and $\{f_k: k\in
\omega\}\subset S$ such that $\{f_k\}$ converge to $f$. For each
$k\in \omega$ there is $\{f^{n}_k : n\in \omega \}\subset B$ such
that $\{f^{n}_k\}$ converge to $f_k$. Since $Y$ is a $QN$-space,
by Theorem 16 in \cite{busu}, there exists an unbounded $\beta\in
\omega^{\omega}$ such that $\{f^{n}_{\beta(n)}\}$ converge to $f$
on $Y$. It follows that $\{f^{n}_{\beta(n)}\}$ converge to $f$ on
$X$ and $[B]_{seq}=B(X)$.

$(2)\Leftrightarrow(3)$. By Theorem \ref{th445} and Theorem
\ref{th37}.

$(3)\Rightarrow(4)$ is immediate.

$(4)\Rightarrow(2)$. Let $B(X)$ $\models$ $S_{1}(\mathcal{S},
\Gamma_x)$ and $B(X)$ be a sequentially separable.  Let
$\{\mathcal{U}_i\}\subset \mathcal{B}_{\Gamma}$ and
$\mathcal{S}=\{h_m\}_{m\in \omega}$ be a countable sequentially
dense subset of $B(X)$. For each $i\in \omega$ we consider a
countable sequentially dense subset $\mathcal{S}_i$ of $B(X)$ and
$\mathcal{U}_i=\{U^{m}_i\}_{m\in \omega}$ where

$\mathcal{S}_i =\{f^m_i\}:=\{ f^m_i\in B(X) : f^m_i\upharpoonright
U^{m}_i=h_m$ and $f^m_i\upharpoonright (X\setminus U^{m}_i)=1$ for
$m \in \omega \}$.

Since $\mathcal{F}_i=\{U^{m}_i\}_{m\in \omega}$ is a
$\gamma$-cover of $X$ and $\mathcal{S}$ is a countable
sequentially dense subset of $B(X)$, we have that $\mathcal{S}_i$
is a  countable sequentially dense subset of $B(X)$ for each $i\in
\omega$.

Let $h\in B(X)$, there is a sequence $\{h_{m_s}\}_{s\in
\omega}\subset \mathcal{S}$ such that $\{h_{m_s}\}_{s\in \omega}$
converge to $h$.
 Let $K$ be a finite subset of $X$, $\epsilon>0$ and $W=<h, K,\epsilon>$ be
a base neighborhood of $h$, then there is a number $m_0$ such that
$K\subset U^{m}_i$ for $m>m_0$ and $h_{m_s}\in W$ for $m_s>m_0$.
Since $f^{m_s}_i\upharpoonright K= h_{m_s}\upharpoonright K$ for
each $m_s>m_0$, $f^{m_s}_i\in W$ for each $m_s>m_0$. It follows
that a sequence $\{f^{m_s}_i\}_{s\in \omega}$ converge to $h$.

By $B(X)$ $\models$ $S_{1}(\mathcal{S}, \Gamma_x)$, there is a
sequence $\{f^{m(i)}_{i}: i\in\omega\}$ such that for each $i$,
$f^{m(i)}_{i}\in \mathcal{S}_i$, and $\{f^{m(i)}_{i}: i\in\omega
\}$ is an element of $\Gamma_0$.

Consider a sequence $\{U^{m(i)}_{i}: i\in \omega\}$.

(a). $U^{m(i)}_{i}\in \mathcal{U}_{i}$.

(b). $\{U^{m(i)}_{i}: i\in \omega\}$ is a $\gamma$-cover of $X$.

Let $K$ be a finite subset of $X$ and $U=<$ $\bf{0}$ $, K,
\frac{1}{2}>$ be a base neighborhood of $\bf{0}$, then there is
$j_0\in \omega$ such that $f^{m(i)_{j}}_{i_{j}}\in U$ for each
$j>j_0$. It follows that $K\subset U^{m(i)_{j}}_{i_{j}}$ for each
$j>j_0$. We thus get $X$ $\models$ $S_{1}(\mathcal{B}_{\Gamma},
\mathcal{B}_{\Gamma})$. By Theorem \ref{th37}, $X$ $\models$ $O$.

$(1)\Rightarrow(5)\Rightarrow(6)$ and
$(4)\Rightarrow(7)\Rightarrow(8)$ are immediate.

$(8)\Rightarrow(2)$. $B(X)$ $\models$
$U_{fin}(\mathcal{S},\Gamma_x)$ $\Rightarrow$  $X$ $\models$
$U_{fin}(\mathcal{B}_{\Gamma}, \mathcal{B}_{\Gamma})$ is proved
similarly to the implication $(4)\Rightarrow(2)$. By Theorem 1 in
\cite{scts},
$S_1(\mathcal{B}_{\Gamma},\mathcal{B}_{\Gamma})\Leftrightarrow
S_{fin}(\mathcal{B}_{\Gamma},\mathcal{B}_{\Gamma})\Leftrightarrow
U_{fin}(\mathcal{B}_{\Gamma},\mathcal{B}_{\Gamma})$.

$(6)\Rightarrow(2)$. Similarly to the implication
$(1)\Rightarrow(2)$.

\end{proof}

\section{$B(X)$ $\models$ $S_{1}(\mathcal{A},\mathcal{A})$}

\begin{definition} A set $A\subseteq B(X)$ will
be called {\it $n$-dense} in $B(X)$, if for each $n$-finite set
$\{x_1,...,x_n\}\subset X$ such that $x_i\neq x_j$ for $i\neq j$
and an open sets $W_1,..., W_n$ in $\mathbb{R}$ there is $f\in A$
such that $f(x_i)\in W_i$ for $i\in \overline{1,n}$.

\end{definition}

Obviously, that if $A$ is a $n$-dense set of $B(X)$ for each $n\in
\omega$ then $A$ is a dense set of $B(X)$.

For a space $B(X)$ we denote:

$\mathcal{A}_n$
--- the family of a countable $n$-dense subsets of $B(X)$.

If $n=1$, then we denote $\mathcal{A}$ instead of $\mathcal{A}_1$.

\begin{definition} Let $f\in B(X)$. A set $B\subseteq B(X)$ will
be called {\it $n$-dense} at point $f$, if for each $n$-finite set
$\{x_1,...,x_n\}\subset X$ and $\epsilon>0$ there is $h\in B$ such
that $h(x_i)\in (f(x_i)-\epsilon, f(x_i)+\epsilon)$ for $i\in
\overline{1,n}$.
\end{definition}

Obviously, that if $B$ is a $n$-dense at point $f$ for each $n\in
\omega$ then $f\in \overline{B}$.

For a space $B(X)$ we denote:

$\mathcal{B}_{n,f}$
--- the family of a countable $n$-dense at point $f$ subsets of $B(X)$.

If $n=1$, then we denote $\mathcal{B}_f$ instead of
$\mathcal{B}_{1,f}$.

 Let $\mathcal{B}$ be a Baire cover of $X$ and $n\in \mathbb{N}$.

$\bullet$ $\mathcal{B}$ is an $n$-cover of $X$ if for each
$F\subset X$ with $|F|\leq n$, there is $B\in \mathcal{B}$ such
that $F \subset B$.


\begin{theorem}\label{th243} For a   space $X$, the following statements are
equivalent:

\begin{enumerate}

\item  $B(X)$ $\models$ $S_{1}(\mathcal{A},\mathcal{A})$;

\item $X$ $\models$ $S_{1}(\mathcal{B}, \mathcal{B})$;

\item $B(X)$ $\models$ $S_{1}(\mathcal{B}_f,\mathcal{B}_f)$;

\item  $B(X)$ $\models$ $S_{1}(\mathcal{A},\mathcal{B}_f)$;





\end{enumerate}

\end{theorem}


\begin{proof} $(1)\Rightarrow(2)$.  Let $\{\mathcal{B}_n : n\in \omega\}$ be a
sequence of countable Baire covers of $X$. We set $A_n=\{f\in
B(X): f\upharpoonright (X\setminus U)=1$ and $f\upharpoonright
U=q$ for  $U\in \mathcal{B}_n$ and $q\in \mathbb{Q}\}$. It is not
difficult to see that each $A_n$ is countable $1$-dense subset of
$B(X)$ since each $\mathcal{B}_n$ is a cover of $X$.

 By the assumption there exist $f_n\in A_n$ such that
$\{f_n : n\in \omega\}\in \mathcal{A}$.

 For each $f_n$ we
take $U_n\in \mathcal{B}_n$ such that
$f_n\upharpoonright(X\setminus U_n)=1$.

 Set $\mathcal{U}=\{ U_n : n\in \omega\}$. For $x\in X$ we consider the basic open neighborhood
of $\bf{0}$ $[x, W]$, where $W=(-\frac{1}{2},\frac{1}{2})$.

 Note that there is $m\in \omega$ such that
$[x, W]$ contains $f_m\in \{f_n : n\in \omega\}$. This means $x\in
 U_m$. Consequently $\mathcal{U}$ is cover
of $X$.

$(2)\Rightarrow(3)$. Let $B_n\in \mathcal{B}_f$ for each $n\in
\omega$. We renumber $\{B_n\}_{n\in \omega}$ as
$\{B_{i,j}\}_{i,j\in \omega}$.  Since $B(X)$ is homogeneous, we
may think that $f=\bf{0}$.  We set
$\mathcal{U}_{i,j}=\{g^{-1}(-1/i, 1/i) : g\in B_{i,j}\}$ for each
$i,j\in \omega$. Since $B_{i,j}\in \mathcal{B}_0$,
$\mathcal{U}_{i,j}$ is a Baire cover of $X$ for each $i,j\in
\omega$. In case the set $M=\{i\in \omega: X\in \mathcal{U}_{i,j}
\}$ is infinite, choose $g_{m}\in B_{m,j}$ $m\in M$ so that
$g^{-1}(-1/m, 1/m)=X$, then $\{g_m : m\in \omega\}\in
\mathcal{B}_f$.

So we may assume that there exists $i'\in \omega$ such that for
each $i\geq i'$ and $g\in B_{i,j}$ $g^{-1}(-1/i, 1/i)$ is not $X$.

For the sequence $\mathcal{V}_i=\{\mathcal{U}_{i,j} : j\in
\omega\}$ of Baire covers there exist $f_{i,j}\in B_{i,j}$ such
that $\mathcal{U}_i=\{f^{-1}_{i,j}(-1/i,1/i):  j\in \omega\}$ is a
cover of $X$.  Let $[x, W]$ be any basic open neighborhood of
$\bf{0}$, where $W=(-\epsilon, \epsilon)$, $\epsilon>0$. There
exists $m\geq i'$ and $j\in \omega$  such that $1/m<\epsilon$ and
$x\in f^{-1}_{m,j}(-1/m, 1/m)$. This means $\{f_{i,j}: i,j\in
\omega\}\in \mathcal{B}_f$.

$(3)\Rightarrow(4)$ is immediate.

$(4)\Rightarrow(1)$. Let $A_n\in \mathcal{A}$ for each $n\in
\omega$. We renumber $\{A_n\}_{n\in \omega}$ as
$\{A_{i,j}\}_{i,j\in \omega}$. Renumber the rational numbers
$\mathbb{Q}$ as $\{q_i : i\in \omega\}$.  Fix $i\in\omega$. By the
assumption there exist $f_{i,j}\in A_{i,j}$ such that $\{f_{i,j} :
j\in \omega\}\in \mathcal{A}_{q_i}$ where $q_i$  is the constant
function to $q_i$. Then $\{f_{i,j} : i,j\in \omega\}\in
\mathcal{A}$.







\end{proof}

\section{$B(X)$ $\models$ $S_{1}(\mathcal{S},\mathcal{A})$}

Well-known (see in \cite{scts}) that
$S_1(\mathcal{B}_{\Gamma},\mathcal{B})\Leftrightarrow
S_{fin}(\mathcal{B}_{\Gamma},\mathcal{B})\Leftrightarrow
U_{fin}(\mathcal{B}_{\Gamma},\mathcal{B})$.

\begin{theorem}\label{th194} For a Tychonoff space $X$, the following statements are
equivalent:

\begin{enumerate}

\item $B(X)$ $\models$ $S_{1}(\Gamma_x,\mathcal{B}_f)$;

\item $X$ $\models$ $S_{1}(\mathcal{B}_{\Gamma}, \mathcal{B})$.

\end{enumerate}

\end{theorem}

\begin{proof} $(1)\Rightarrow(2)$. Let $\{\mathcal{U}_i\}\subset
 B_{\Gamma}$. For each $i\in \omega$ we consider the
 set $\mathcal{S}_i:=\{ f\in B(X) : f\upharpoonright
U=0$ and $f\upharpoonright (X\setminus U)=1$ for $U\in
\mathcal{U}_i \}$.

Since $\mathcal{U}_i$ is a $\gamma$-cover of $X$, we have that
$\mathcal{S}_i$ converge to ${\bf 0}$, i.e. $\mathcal{S}_i\in
\Gamma_0$ for each $i\in \omega$.

Since  $B(X)$ $\models$ $S_{1}(\Gamma_x,\mathcal{B}_f)$, there is
a sequence $\{f_{i}\}_{i\in\omega}$ such that for each $i$,
$f_{i}\in \mathcal{S}_i$, and $\{f_{i} : i\in\omega\}\in
\mathcal{B}_{\bf 0}$.

Consider $\mathcal{V}=\{U_i : U_i\in \mathcal{U}_i$ such that
$f_i\upharpoonright U_i=0$ and $f_i\upharpoonright (X\setminus
U_i)=1\} $. Let $x\in X$ and $W=[x,(-1,1)]$ be a neighborhood of
$\bf{0}$, then there exists $i_0\in \omega$ such that $f_{i_0}\in
W$ .

 It follows that $x\in U_{i_0}$ and $\mathcal{V}\in \mathcal{B}$. We thus get
$X$ $\models$ $S_{1}(\mathcal{B}_{\Gamma}, \mathcal{B})$.

$(2)\Rightarrow(1)$.  Fix $\{S_n : n\in \omega\}\subset \Gamma_0$.
We renumber $\{S_n : n\in \omega\}$ as $\{S_{i,j}: i,j\in
\omega\}$.

For each $i,j\in \omega$ and $f\in S_{i,j}$, we put
$U_{i,j,f}=\{x\in X : |f(x)|<\frac{1}{i+j}\}$.

Each $U_{i,j,f}$  is a Baire set of $X$. Let $\mathcal{U}_{i,j}=\{
U_{i,j,f} : f\in S_{i,j}\}$. So without loss of generality, we may
assume $U_{i,j,f}\neq X$ for each $i,j\in \omega$ and $f\in
S_{i,j}$.

We can easily check that the condition $S_{i,j}\in \Gamma_0$
implies that $\mathcal{U}_{i,j}$ is a $\gamma$-cover of $X$.

Since  $X$ $\models$ $S_{1}(\mathcal{B}_{\Gamma}, \mathcal{B})$
for each $j\in \omega$ there is a sequence $\{U_{i,j,f_{i,j}} :
i\in \omega\}$ such that for each $i$, $U_{i,j,f_{i,j}}\in
\mathcal{U}_{i,j}$, and $\{U_{i,j,f_{i,j}} : i\in \omega\}\in
\mathcal{B}$. Claim that $\{f_{i,j}:i,j\in \omega\}\in
\mathcal{B}_f$. Let $x\in X$, $\epsilon>0$, and $W=[x,(-\epsilon,
\epsilon)]$ be a base neighborhood of $\bf{0}$, then there exists
$j'\in \omega$ such that $\frac{1}{1+j'}<\epsilon$. It follow that
there is $f_{i,j'}\in \{f_{i,j}:i,j\in \omega\}$ such that
$f_{i,j'}(x)\in (-\epsilon, \epsilon)$.

So $B(X)$ $\models$ $S_{1}(\Gamma_x,\mathcal{B}_f)$.

\end{proof}

\begin{theorem}\label{th173} For a   space $X$, the following statements are
equivalent:

\begin{enumerate}

\item  $B(X)$ $\models$ $S_{1}(\mathcal{S},\mathcal{A})$ and is
sequentially separable;

\item $X$ $\models$ $S_{1}(\mathcal{B}_{\Gamma}, \mathcal{B})$ and
$X$ $\models$ $O$;

\item $B(X)$ $\models$ $S_{1}(\Gamma_x,\mathcal{B}_f)$ and is
sequentially separable;

\item  $B(X)$ $\models$ $S_{1}(\mathcal{S},\mathcal{B}_f)$ and is
sequentially separable.

\end{enumerate}

\end{theorem}

\begin{proof} $(1)\Rightarrow(2)$. Let $\{\mathcal{U}_i\}\subset \mathcal{B}_{\Gamma}$ and
$\mathcal{S}=\{h_m\}_{m\in \omega}$ be a countable sequentially
dense subset of $B(X)$. Let $\mathcal{U}_i=\{ U^{m}_i\}_{m\in
\omega}$.

For each $i\in \omega$ we consider a countable sequentially dense
subset $\mathcal{S}_i$ of $B(X)$ where

$\mathcal{S}_i =\{ f^m_i\in B(X) : f^m_i\upharpoonright
U^{m}_i=h_m$ and $f^m_i\upharpoonright (X\setminus U^{m}_i)=1$ for
$m \in \omega \}$.

Since $\mathcal{U}_i$ is a $\gamma$-cover of  $X$ and
$\mathcal{S}$ is a countable sequentially dense subset of $B(X)$,
we have that $\mathcal{S}_i$ is a countable sequentially dense
subset of $B(X)$ for each $i\in \omega$.  Let $h\in B(X)$, there
is a sequence $\{h_{m_s}: s\in \omega\}\subset \mathcal{S}$ such
that $\{h_{m_s}\}_{s\in \omega}$ converge to $h$.
 Let $K$
be a finite subset of $X$, $\epsilon>0$ and $W=<h, K,\epsilon>$ be
a base neighborhood of $h$, then there is a number $m_0$ such that
$K\subset U^{m}_i$ for $m>m_0$ and $h_{m_s}\in W$ for $m_s>m_0$.
Since $f^{m_s}_i\upharpoonright K= h_{m_s}\upharpoonright K$ for
each $m_s>m_0$, $f^{m_s}_i\in W$ for each $m_s>m_0$. It follows
that a sequence $\{f^{m_s}_i\}_{s\in \omega}$ converge to $h$.

By $B(X)\in S_{1}(\mathcal{S},\mathcal{A})$, there is a sequence
$\{f^{m(i)}_{i}\}_{i\in\omega}$ such that for each $i$,
$f^{m(i)}_{i}\in \mathcal{S}_i$, and $\{f^{m(i)}_{i}: i\in\omega
\}$ is an element of $\mathcal{A}$.

Consider a sequence $\{U^{m(i)}_{i}: i\in \omega\}$.

(a). $U^{m(i)}_{i}\in \mathcal{U}_{i}$.

(b). $\{U^{m(i)}_{i}: i\in \omega\}$ is a Baire cover of $X$.

Let $K$ be a finite subset of $X$ and $U=<$ $\bf{0}$ $, K,
\frac{1}{2}>$ be a base neighborhood of $\bf{0}$, then there is
$f^{m(i)_{j_0}}_{i_{j_0}}\in U$ for some $j_0\in \omega$. It
follows that $K\subset U^{m(i)_{j_0}}_{i_{j_0}}$. We thus get

$X$ $\models$ $S_{1}(\mathcal{B}_{\Gamma}, \mathcal{B})$.

$(2)\Leftrightarrow(3)$. By Theorem \ref{th194} and Theorem
\ref{th37}.

$(3)\Rightarrow(4)$ is immediate.

$(4)\Rightarrow(1)$. Suppose that $B(X)$ is sequentially separable
and $B(X)$ $\models$ $S_{1}(\mathcal{S},\mathcal{B}_f)$. Let
$D=\{d_n: n\in \omega \}$ be a countable sequentially dense
subspace of $B(X)$. Given a sequence of sequentially dense
subspace of $B(X)$, enumerate it as $\{S_{n,m}: n,m \in \omega
\}$. For each $n\in \omega$, pick $d_{n,m}\in S_{n,m}$ so that
${\{d_{n,m}: m\in \omega\}}\in \mathcal{B}_{d_n}$. Then
$\{d_{n,m}: m,n\in \omega\}\in \mathcal{A}$.

\end{proof}

We can summarize the relationships between considered notions in
next diagrams.

\newpage

{\begin{figure}
\begin{center}

\ingrw{90}{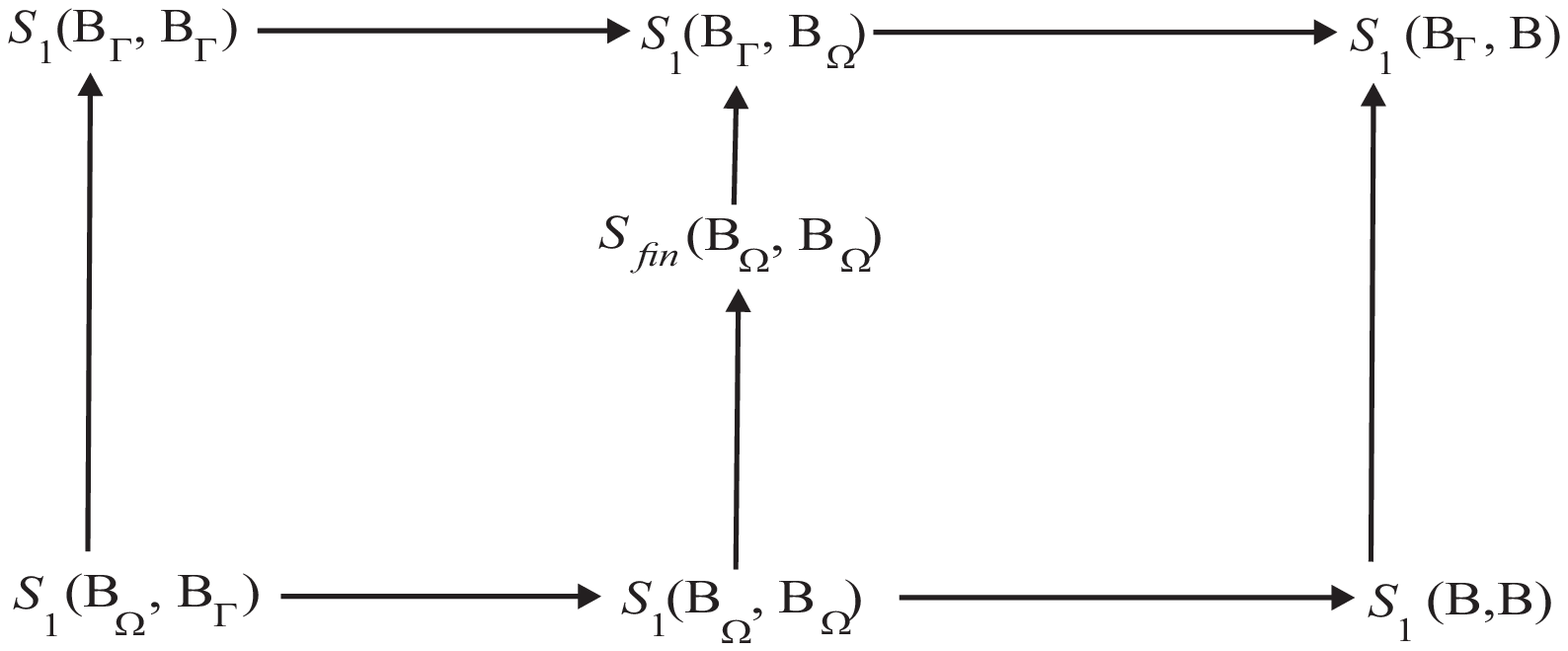}

Figure 2. The Scheepers Diagram in the Baire case.

\end{center}
\end{figure}}

\bigskip
{\begin{figure}
\begin{center} \ingrw{90}{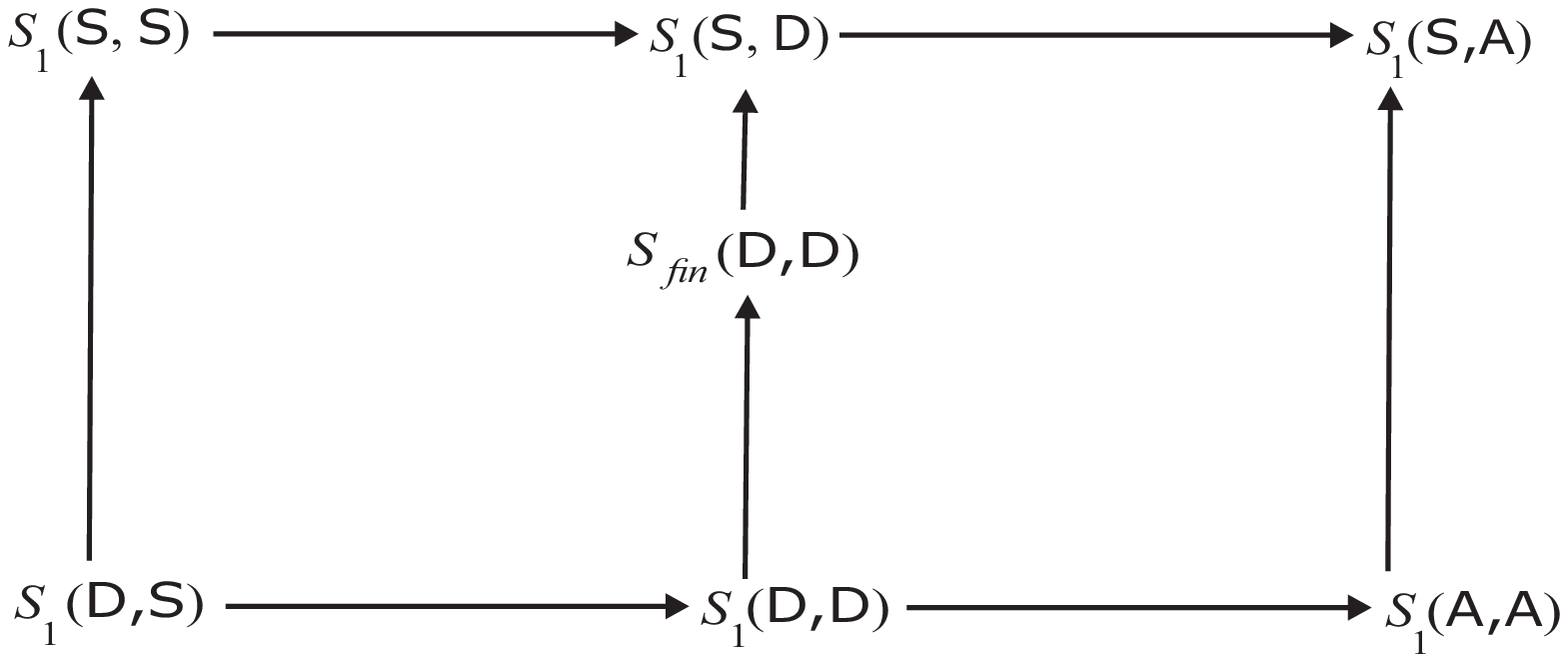}

\medskip

Figure 3. The Diagram of selection principles of selectors for
sequences of countable dense and countable sequentially subsets of
$B(X)$ corresponding to selection principles of Figure 2.

\end{center}

\end{figure}}

\section{Critical cardinalities}

 The critical cardinalities in the Scheepers Diagram in the Baire case are equal to the critical cardinalities of selectors for
sequences of countable dense and countable sequentially subsets of
$B(X)$.

\medskip

$\bullet$ According to Corollary 41 in \cite{scts}, Theorem
\ref{th22} and Theorem \ref{th444} it is consistent that there is
a set of real numbers $X$ such that $B(X)$ with property
$S_{1}(\mathcal{D},\mathcal{D})$, but not property
$S_{1}(\mathcal{S},\mathcal{S})$.

 $\bullet$ According to Theorem 32 in \cite{scts}, Theorem \ref{th243} and Theorem \ref{th111} it is consistent that there is a set
of real numbers $X$ such that $B(X)$ in
$S_{1}(\mathcal{A},\mathcal{A})$ which is not in
$S_{1}(\mathcal{S},\mathcal{D})$.

 $\bullet$ According to Theorem 43 in \cite{scts}, Theorem \ref{th243}, Theorem \ref{th22} and Theorem \ref{th444} it is consistent that there is a set
of real numbers $X$ such that $B(X)$ in
$S_{1}(\mathcal{S},\mathcal{S})$ and not in either of
$S_{fin}(\mathcal{D},\mathcal{D})$ or
$S_{1}(\mathcal{A},\mathcal{A})$.

 $\bullet$ According to Theorem 27 in \cite{scts} and Theorem \ref{th35} the minimal cardinality of a set of
real numbers $X$ such that $B(X)$ not having property
$S_{fin}(\mathcal{D},\mathcal{D})$  is $\mathfrak{d}$ while the
minimal cardinality of a set of real numbers $X$ such that $B(X)$
not having property $S_{1}(\mathcal{A},\mathcal{A})$ is
$cov(\mathcal{M})$. Since it is consistent that $cov(\mathcal{M})<
\mathfrak{d}$, it is consistent that none of the arrows starting
at the bottom row of Fig. 3 is reversible.

\bigskip

For a collection $\mathcal{J}$ of spaces of all Baire functions,
defined on  Tychonoff spaces $X$ with $iw(X)=\aleph_0$,  let
$nonB(\mathcal{J})$ denote the minimal cardinality for $X$ which
$B(X)$ is not a member of $\mathcal{J}$.

\begin{theorem} For a collection $B(X)$ of spaces of all Baire functions,
defined on  Tychonoff spaces $X$ with $iw(X)=\aleph_0$,

(1) $nonB(S_{1}(\mathcal{D},\mathcal{S}))=\mathfrak{p}$.

(2) $nonB(S_{1}(\mathcal{S},\mathcal{S}))=\mathfrak{b}$.

(3)
$nonB(S_{fin}(\mathcal{D},\mathcal{D}))=nonB(S_{1}(\mathcal{S},\mathcal{D}))=
nonB(S_{1}(\mathcal{S},\mathcal{A}))=\mathfrak{d}$.

(4) $nonB(S_{1}(\mathcal{D},\mathcal{D})) =
nonB(S_{1}(\mathcal{A},\mathcal{A}))=cov(\mathcal{M})$.

\end{theorem}

\begin{proof}

By Theorem 27 \cite{scts},

(1) $non(S_1(\mathcal{B}_{\Omega},\mathcal{B}_{\Gamma}
))=\mathfrak{p}$. By Theorem \ref{th212},
$nonB(S_{1}(\mathcal{D},\mathcal{S}))=\mathfrak{p}$.

(2) $non(S_1(\mathcal{B}_{\Gamma} ,\mathcal{B}_{\Gamma}
))=\mathfrak{b}$. By Theorem \ref{th444},
$nonB(S_{1}(\mathcal{S},\mathcal{S}))=\mathfrak{b}$.

(3)
$non(S_{fin}(\mathcal{B}_{\Omega},\mathcal{B}_{\Omega}))=\mathfrak{d}$.
By Theorem \ref{th35},
$nonB(S_{fin}(\mathcal{D},\mathcal{D}))=\mathfrak{d}$.

(4)$non(S_1(\mathcal{B}_{\Gamma}
,\mathcal{B}_{\Omega}))=\mathfrak{d}$. By Theorem \ref{th111},
$nonB(S_{1}(\mathcal{S},\mathcal{D}))=\mathfrak{d}$.

(5)$non(S_1(\mathcal{B}_{\Gamma} ,\mathcal{B}))= \mathfrak{d}$. By
Theorem \ref{th173},
$nonB(S_{1}(\mathcal{S},\mathcal{A}))=\mathfrak{d}$.

(6)
$non(S_1(\mathcal{B}_{\Omega},\mathcal{B}_{\Omega}))=cov(\mathcal{M})$.

By Theorem \ref{th22},
$nonB(S_{1}(\mathcal{D},\mathcal{D}))=cov(\mathcal{M})$.

(7) $non(S_1(\mathcal{B},\mathcal{B}))=cov(\mathcal{M})$.

By Theorem \ref{th243},
$nonB(S_{1}(\mathcal{A},\mathcal{A}))=cov(\mathcal{M})$.
\end{proof}


\bibliographystyle{model1a-num-names}
\bibliography{<your-bib-database>}

\begin{thebibliography}{10}


\bibitem{arh0}
A.V. Arhangel'skii, \textit{The frequency spectrum of a
topological space and the classification of spaces}, Soviet Math.
Dokl. 13, (1972), 1186--1189.





\bibitem{arh}
A.V. Arhangelskii, \textit{Hurewicz spaces, analytic sets and fan
tightness of function spaces}, Soviet Math. Dokl. 33, (1986),
396--399.




\bibitem{bbm1}
A. Bella, M. Bonanzinga, M. Matveev, \textit{Variations of
selective separability}, Topology and its Applications, 156,
(2009), 1241--1252.




\bibitem{busu}
L. Bukovsk$\acute{y}$, J. $\check{S}$upina, \textit{Sequence
selection principles for quasi-normal convergence}, Topology and
its Applications, 159, (2012), p.283--289.

\bibitem{do}
E.K. van Douwen, \textit{The integers and topology}, in: Handbook
of Set-Theoretic Topology, North-Holland, Amsterdam, (1984).


\bibitem{jmss}
W. Just, A.W. Miller, M. Scheepers, P.J. Szeptycki, \textit{The
combinatorics of open covers, II}, Topology and its Applications,
73, (1996), 241--266.

\bibitem{nob}
 N. Noble, \textit{The density character of functions spaces}, Proc. Amer. Math. Soc. (1974), V.42, is.I.-P., 228--233.


\bibitem{os1}
A.V. Osipov, \textit{The application of selection principles in
the study of the properties of function spaces}, to appear.

\bibitem{osa}
A.V. Osipov, \textit{Classification of selectors for sequences of
dense sets of $C_p(X)$}, to appear.


\bibitem{obp}
\textit{Strongly sequentially separable function spaces, via
selection principles}, to appear.





\bibitem{ps}
A.V. Pestrjkov, \textit{O prostranstvah berovskih funktsii},
Issledovanij po teorii vipuklih mnogestv i grafov, Sbornik
nauchnih trudov, Sverdlovsk, Ural'skii Nauchnii Center, (1987),
p.53--59.


\bibitem{sak}
M. Sakai, \textit{Property $C^{''}$ and function spaces}, Proc.
Amer. Math. Soc. 104 (1988), 917–919.


\bibitem{sash}
M. Sakai, M. Scheepers, \textit{The combinatorics of open covers},
Recent Progress in General Topology III, (2013), Chapter, p.
751--799.





\bibitem{sch4}
M. Scheepers, \textit{A sequential property of $C_p(X)$ and a
covering property of Hurewicz}, Proceedings of the American
Mathematical Society, 125, (1997), p. 2789--2795.




\bibitem{scts}
M. Scheepers, B. Tsaban, \textit{The combinatorics of Borel
covers}, Topology and its Applications, 121, (2002), p.357--382.



\bibitem{tszd}
B. Tsaban, L. Zdomskyy, \textit{Hereditarily Hurewicz spaces and
Arhangel'ski$\breve{i}$ sheaf amalgamations}, Journal of the
European Mathematical Society, 12, (2012), 353--372.







\end{thebibliography}


\bibliographystyle{plain}





\end{document}